\documentclass[11pt,a4paper]{article}
\usepackage[utf8]{inputenc}
\usepackage[english]{babel}
\usepackage{geometry}
\usepackage{amsmath}
\usepackage{amssymb}
\usepackage{mathtools}
\usepackage{booktabs}
\usepackage{hyperref}
\usepackage{amsthm}
\usepackage{amsmath}
\usepackage{enumitem}
\usepackage{tikz}
\usepackage{nccmath}
\usepackage{bm}

\usepackage{tikz-cd}
\usepackage{tikz}
\usepackage{bbm}
\mathtoolsset{showonlyrefs}

\usepackage{dsfont}
\usepackage{hyperref}

\usepackage{mathrsfs}
\usepackage{braket}
\usepackage{mathtools}
\usepackage{booktabs}
\usepackage{caption}
\usepackage{cancel}
\usepackage{subfig}
\usepackage{epstopdf}

\numberwithin{equation}{section}
\newcommand{\numberset}{\mathbb}
\newcommand{\N}{\numberset{N}}
\newcommand{\R}{\numberset{R}}

\newtheorem{lemma}{Lemma}[section]
\newtheorem{theorem}{Theorem}

\newtheorem{prp}[lemma]{Proposition}
\newtheorem{cor}[lemma]{Corollary}
\newtheorem{defn}[lemma]{Definition}
\newtheorem{claim}[lemma]{Claim}
\theoremstyle{remark}
\newtheorem*{remark}{\textbf{Remark}}

\usepackage[style=alphabetic]{biblatex}
\renewbibmacro{in:}{%
  \ifentrytype{article}
    {}
    {\bibstring{in}%
     \printunit{\intitlepunct}}}
\usepackage{tikz}
\usetikzlibrary{decorations.pathreplacing}

\title{Sharp non-explicit blow-up profile for perturbed nonlinear heat equations with gradient terms}
\author{Maissâ Boughrara\vspace{1em} \\
Université Sorbonne Paris Nord, LAGA (UMR 7539),\\ F-93430 Villetaneuse, France.} 
\date{}

\begin{document}
\maketitle

\begin{abstract}
We consider a class of blow-up solutions for perturbed nonlinear heat equations involving gradient terms. We first prove the single point blow-up property for this equation and determine its final blow-up profile. We also give a sharper description of its blow-up behaviour, where we take as a profile some suitably chosen solution of the unperturbed semilinear heat equation. The proof relies on selfsimilar variables with involved arguments to control the gradient term.
\end{abstract}

\section{Introduction and statement of the results}
We are concerned with  blow-up phenomenon for the following perturbed nonlinear heat equation:
\begin{equation}\label{eq:1}
\begin{split}
u_t=\Delta u+|u|^{p-1}u+h(u,\nabla u),\\
u(.,0)=u_0\in W^{1,\infty}(\R^N),
\end{split}
\end{equation}
where $1<p$, if $N\geq 3$ then $(N-2)p<(N+2)$ , and 
\begin{equation}\label{condition h}
h\in \mathcal{C}^1(\R^{N+1}), |h(u,v)|\leq C(|u|^{\gamma}+|v|^{\overline{\gamma}}+1),
\end{equation}
with $0\leq \gamma<p, 0\leq \overline{\gamma}<\frac{2p}{p+1}$ and $C>0$.

The unperturbed case (i.e $h\equiv0$) has been extensively studied in the literature, and non bibliography can be exhaustive. We simply refer to the book by Quittner and Souplet in \cite{SQ2019} and the references therein. In this paper, our aim is to consider the perturbed case, and to address the question of refining the blow-up profile. Using standard contraction arguments, the Cauchy problem for equation \eqref{eq:1} is well posed in $W^{1,\infty}(\R^N)$, locally in time. Therefore, the maximal solution $u(t)$ exists for all $t \in [0, T )$, with $T \leq +\infty$.
Then, either it is a global solution, or it exists only up to some finite time $T<+\infty$ called the blow-up time, and in this case the solution blows up in the following sense:
$$\|u(t)\|_{W^{1,\infty}(\R^N)}\longrightarrow +\infty \text{ when } t\rightarrow T.$$
We say that $a$ is a blow-up point of $u$, if there exists a sequence $(x_n,t_n)\in \R^{N}\times [0,T)$ such that
$$(x_n,t_n)\underset{n\rightarrow +\infty}{\longrightarrow} (a,T)\text{ and }|(x_n,t_n)|\underset{n\rightarrow +\infty}{\longrightarrow} +\infty.$$

Some authors investigated the asymptotic behaviour of blow-up solutions. We may cite Souplet, Tayachi and Weissler in \cite{STW1996}, Alfonsi and Weissler in \cite{AW1992}. In particular, in \cite{EZ2011}, Ebde and Zaag constructed a solution of \eqref{eq:1} which blows up at some point $a$. Moreover, they showed that the constructed solution obeys the following blow-up profile:

\begin{equation}\label{behaviour1}
\left|\left|(T-t)^{\frac{1}{p-1}}u(.,t)-f\left(\frac{.-a}{\sqrt{(T-t)|\log(T-t)|}}\right)\right|\right|_{W^{1,\infty}(\R^N)}=O\left(\frac{1}{\sqrt{|\log(T-t)|}}\right),
\end{equation}
where
\begin{equation}\label{expression f}
f(z)=\left(p-1+\frac{(p-1)^2}{4p}|z|^2\right)^{-\frac{1}{p-1}}.
\end{equation}

However, two issues remain unsolved following that work:
\begin{itemize}
\item Is $a$ the only blow-up point?
\item Can we define a notion of final blow-up profile $u_*$, defined for all $x\in \R^N\backslash \{a\}$ such that
\begin{equation}\label{u_*}
u_*(x)=\underset{t\rightarrow T}{\lim}u(x,t)?
\end{equation}
Can we determine its equivalent near the blow-up point?
\end{itemize}

As a matter of fact, our first aim in this paper is to answer those open questions. For this matter we define the set $S_{a,T,h}$.
\begin{defn}
   For a given $(a,T)\in \R^N\times \R^+$ and $h$ verifies \eqref{condition h}, we define $S_{a,T,h}$ as the set of the solutions $u$ of \eqref{eq:1} with perturbation term $h$ and which blows up at time $T$ and at the point $a$, with $q(x,s)\in V_A(s)$ for all $s\geq s_0$ and for a certain $A>0$, where $q$ and $V_A(s)$ are defined later in \eqref{equality w and q} and in Propostion \ref{def V_A} respectively.
\end{defn}
In this definition, we don't require $a$ to be the only blow-up point. As a matter of fact, one of our main results in this paper is to prove that this is the case of all $u\in S_{a,T,h}.$
\begin{remark}
    The solution constructed in \cite{EZ2011} belongs to $S_{a,T,h}$.
\end{remark}
Our first result answers the questions mentionned earlier. This is the statement:
\begin{theorem}[\textbf{Single-point blow-up and final profile}]\label{Thm 2 g}
Assume $N\geq 1$ and let be $(a,T)\in \R^N\times \R^+$, $h$ verifies \eqref{condition h} and $u\in S_{a,T,h}$.
\begin{enumerate}[label=(\roman*)]
\item It holds that $"a"$ is the only blow-up point.
\item There exists a final profile $u_*$ as in \eqref{u_*}, for all $x\in \R^N\backslash \{a\}$. Moreover, there exists $\epsilon>0$, such that for all $\epsilon\geq |x-a|> 0$\label{Thm 2 ii}
\begin{equation}\label{Thm 2}
\left|u_*(x)-\left(\frac{8p|\log|x-a||}{(p-1)^2|x-a|^{2}}\right)^{\frac{1}{p-1}}\right|\leq C\frac {|\log|x-a||^{\frac{1}{p-1} - \frac{1}{4}}} {|x-a|^{\frac{2}{p-1}}}.
\end{equation}
\end{enumerate}
\end{theorem}

Our next goal is to see whether we can improve the error estimate in \eqref{behaviour1}, or whether we can derive a Taylor expansion of the solution. As it was noted by Zaag in \cite{Z2002}, any attempt of doing such thing will lead to computations trapped in the scales of powers of $\frac{1}{|\log(T-t)|}$, which is a slow decaying variable in comparison to the natural decaying variable $T-t$. This was already noticed formally by Vel{\'a}zquez, Galaktionov, and Herrero in \cite{VGH1991}. In particular, having an expansion in the scale of powers of $T-t$ seemed out of reach, at least if one applies naive iterations to derive higher order terms starting from the first term of order $\frac{1}{|\log(T-t)|}$.


In some sense, the problem comes from the fact that we are linearising our solution around the profile $f$ shown in \eqref{behaviour1} and in \eqref{expression f} which is not sharp enough (The error is in the slow variable scale). The idea is to replace this profile by a sharper version showing a smaller error in the fast variable $T-t$. As already suggested in \cite{Z2002}, a good candidate for that sharper version may be simply the solution of the heat equation with a pure power source (i.e. $h\equiv 0$). However, there is one draw back in this idea: We are replacing an explicit profile by a non-explicit one.

We consider $\hat u$ solution of \eqref{eq:1} with $h\equiv 0$ which blows up at $T=1$ and at the origin. Many proved the existence of such solution, we may cite \cite{MZ1997}, but we will take $u\in S_{0,1,0}$. The following result is our improvement of \eqref{behaviour1}.


\begin{theorem}[\textbf{Sharp non-explicit blow-up profile}]\label{Thm: ch1}
Assume $N\geq 1$ and let be $(a,T)\in \R^N\times \R^+$, $h$ verifies \eqref{condition h} and $u\in S_{a,T,h}$. Then, for all $(x,t)\in \R^N\times [\max(0,T-1), T)$,
\begin{equation}\label{estimation dim N}
|u(x,t)-\hat u(x-a,t-T+1)|\leq C \max \left\{\frac{(T-t)^{-\frac{1}{p-1}}}{|\log(T-t)|}, \frac{|x-a|^{-\frac{2}{p-1}}}{|\log|x-a||^{\frac{1}{4}-\frac{1}{p-1}}} \right\}.
\end{equation}

Moreover, if N=1, then there exists $\lambda>0$ such that:

\begin{equation}\label{estimation dim 1}
\begin{split}
|u(x,t)-\lambda^{\frac{2}{p-1}}&\hat{u}(\lambda (x-a),\lambda^2(t-T)+1)|\\
\leq C \max&\left(\frac{(T-t)^{\frac{1}{2}-\frac{1}{p-1}}}{|\log(T-t)|^{\frac{3}{2}}}, \frac{(T-t)^{\nu-\frac{1}{p-1}}}{|\log(T-t)|^{2-\nu}}\exp\left(C\sqrt{-\log(T-t)}\right),\right.\\
&\ \left.\frac{|x|^{\min(1,2\nu)-\frac{2}{p-1}}}{|\log |x||^{\min(2,\nu)-\frac{1}{p-1}}}, \frac{|x|^{2\nu-\frac{2}{p-1}}}{|\log |x||^{\min(2,\nu)-\frac{1}{p-1}}}\exp\left(C\sqrt{-\log |x|}\right)\right),
\end{split}
\end{equation}
where $\nu=\frac{p-\gamma_0}{p-1}$ and $\gamma_0=\max\{\gamma,\overline{\gamma}(p+1)-p\}$.
\end{theorem}
\begin{remark}
    Note from the line following \eqref{condition h} that $\nu>0$.
\end{remark}
\begin{remark}
In the one-dimensional case, we have this direct consequence which allows us to see better the estimation: 
\begin{equation*}
\begin{split}
|u(x,t)-\lambda^{\frac{2}{p-1}}&
\hat{u}(\lambda (x-a),\lambda^2(t-T)+1)|\\
\leq C \max&\left(\frac{(T-t)^{\frac{1}{2}-\frac{1}{p-1}}}{|\log(T-t)|^{\frac{3}{2}}}, \frac{(T-t)^{\nu-\frac{1}{p-1}-\varepsilon}}{|\log(T-t)|^{2-\nu}},\right.\\
&\ \left.\frac{|x-a|^{\min(1,2\nu)-\frac{2}{p-1}}}{|\log |x-a||^{\min(2,\nu)-\frac{1}{p-1}}}, \frac{|x-a|^{2\nu-\frac{2}{p-1}-\varepsilon}}{|\log |x-a||^{\min(2,\nu)-\frac{1}{p-1}}}\right),
\end{split}
\end{equation*}
for some $\varepsilon>0$. Although the estimation \eqref{estimation dim 1} is sharper than this last one.
\end{remark}

In fact, we will prove a more general result, which implies Theorem \ref{Thm: ch1}. Though, before stating the result, we need to introduce few definitions. Introducing the following similarity variables:

\begin{equation}\label{cv}
\begin{split}
y=\frac{x-a}{\sqrt{T-t}},\ s=-\log(T-t).\\
w_{a,T}(y,s)=(T-t)^{\frac{1}{p-1}}u(x,t).
\end{split}
\end{equation}
Then, \eqref{eq:1} is equivalent to the following equation:
\begin{equation}\label{eq1:cv}
w_s=\Delta w-\frac{1}{2}y.\nabla w-\frac{w}{p-1}+|w|^{p-1}w+e^{-\frac{p}{p-1} s}h(e^{\frac{1}{p-1} s}w,e^{\frac{p+1}{2(p-1)}s}\nabla w).
\end{equation}
We define $q$, for all $y\in \R^N$ and $s\geq s_0=-\log T$, by
\begin{equation}\label{equality w and q}
w(y,s)-f\left(\frac{y}{\sqrt{s}}\right)=\frac{N\kappa}{2ps}+q(y,s),
\end{equation}
where $\kappa=(p-1)^{-\frac{1}{p-1}}$.

\begin{defn}\label{Def S D}
For all $(a,T)\in \R^N\times \R$. We consider the following mappings:

\begin{equation*}
\begin{aligned}
\mathcal{T}_{\xi,\tau} & : & S_{a,T,h} & \rightarrow  S_{a+\xi,T+\tau,h}\\
& \quad & u & \mapsto  (\mathcal{T}_{\xi,\tau}u:(x,t)\mapsto u(x-\xi,t-\tau)).
\end{aligned}
\end{equation*}

\begin{equation*}
\begin{aligned}
\mathcal{D}_{\lambda} & : & S_{a,T,h} & \rightarrow  S_{\lambda^{-1}a,\lambda^{-2}T,h_\lambda}\\
& \quad & u & \mapsto  (\mathcal{D}_{\lambda}u:(x,t)\mapsto \lambda^{\frac{2}{p-1}} u(\lambda x,\lambda^2 t)),
\end{aligned}
\end{equation*}
where $h_\lambda(u,v)=\lambda^{\frac{2p}{p-1}}h(\lambda^{-\frac{2}{p-1}}u,\lambda^{-\frac{p+1}{p-1}}v)$.
\end{defn}

One may see that Theorem \ref{Thm: ch1} is an immediate consequence of the following Theorem:
\medskip

\noindent \textbf{Theorem \ref{Thm: ch1}' (Difference between two blow-up solutions)}
\textit{Assume $N\geq 1$ and consider $u_i\in S_{a_i,T_i,h_i}$, for $i=1,2$. Then, $\mathcal{T}_{a_1-a_2,T_1-T_2}u_2\in S_{a_1,T_1,h_2}$, and for all $(x,t)\in \R^N\times [\max(0,T_1-T_2), T_1)$}, 
\begin{equation}\label{thm estimation dim N}
|u_1(x,t)-\mathcal{T}_{a_1-a_2,T_1-T_2}u_2(x,t)| \text{ verifies the estimation \eqref{estimation dim N}},
\end{equation}
\textit{for $T=T_1$ and $a=a_1$. Moreover, if N=1, then there exists $\lambda>0$ such that $\mathcal{T}_{a_1,T_1}\mathcal{D}_\lambda \mathcal{T}_{-a_2,-T_2} u_2\in S_{a_1,T_1,h_{2,\lambda}}$ and }
\begin{equation}\label{thm estimation dim 1}
|u_1(x,t)-\mathcal{T}_{a_1,T_1}\mathcal{D}_\lambda \mathcal{T}_{-a_2,-T_2} u_2(x,t)| \text{ verifies the estimation \eqref{estimation dim 1}}.
\end{equation}
Let us mention that
\begin{equation}
    \mathcal{T}_{a_1,T_1}\mathcal{D}_\lambda \mathcal{T}_{-a_2,-T_2} u_2(x,t)=u_2(\lambda(x-a_1)+a_2,\lambda^2(t-T_1)+T_2).
\end{equation}

\begin{remark}
    Similar result has been proved previously for the case of $h_i\equiv 0, i=1,2$ by Kammerer and Zaag in \cite{KZ2000}. Our goal is to show that the last term of equation \eqref{eq:1} does not have a large effect on the solution's blow-up  for large time. In fact, our approach in this paper is inspired by the main idea of \cite{KZ2000} and some technical tools used in \cite{GK1989}, \cite{TZ2015} and in \cite{AZ2021}. It might seem easy to show such thing, since our problem is just a simple perturbation of the equation of \cite{KZ2000}. If this is true in the statements, it is not the case for the proof, where we need some involved arguments to control the $h_i(u_i, \nabla u_i)$ term.
\end{remark}

In this paper, we only prove Theorem \ref{Thm: ch1}', knowing that Theorem \ref{Thm: ch1} is an immediate consequence. We proceed in five sections for the proof. Note in particular, the conclusion is in the seventh section. Theorem \ref{Thm 2 g} is done in the eight section together with the section \ref{No blow up threshold}.

\textbf{Acknowledgement:} The author wants to thank Hatem Zaag for useful conversations and comments, who generously provided knowledge and expertise during this project.
Moreover, the author would like to express her gratitude to the referees for the effort they dedicated to reviewing the article and providing valuable suggestions which contributed to the improvement of this publication.

\section{Setting the problem}
Before heading to the proof, we need to set the frame of our work in order to give at least the strategy first.
\subsection{A shrinking set $V_A$ to zero}
We will give more details about the solution constructed in \cite{EZ2011} by defining the set $V_A$. Let $u$ be the solution which follows the profile \eqref{behaviour1}.
\subsubsection{Linearisation in the neighbourhood of the explicite profile}
One may see from \eqref{eq1:cv} that $q$ defined in \eqref{equality w and q} verifies the following equation:
\begin{equation*}
q_s=(\mathcal{L}+V)q+B(q)+R+N(q),
\end{equation*}
where for all $s\geq s_0$,

\begin{equation}\label{def op L}
\begin{split}
&\mathcal{L}=\Delta-\frac{y}{2}.\nabla+1,\\
&V=p\ \varphi^{p-1}-\frac{p}{p-1},\\
&B(q)=|\varphi+q|^{p-1}(\varphi+q)-\varphi^q-p\varphi^{p-1}q,\\
&R=\Delta \varphi -\frac{1}{2}y.\nabla \varphi-\frac{\varphi}{p-1}+\varphi^q-\partial_s\varphi,\\
&N(q)=e^{-\frac{p}{p-1} s}h(e^{\frac{1}{p-1} s}(\varphi+ q),e^{\frac{p+1}{2(p-1)}s}(\nabla\varphi+\nabla q)),
\end{split}
\end{equation}
with 
$$\varphi(y,s)=f\left(\frac{y}{\sqrt{s}}\right)+\frac{\kappa}{2ps},$$
where $f$ and $\kappa$ are given in \eqref{expression f} and \eqref{equality w and q} respectively.

The study of $q$ is equivalent to the study of the linear term $\mathcal{L}+V$, the quadratic term $B(q)$, the remain terms $R$ and $N(q)$. For more details, we will simply refer to the work of Ebde and Zaag in \cite{EZ2011}.
\subsubsection{Decomposition of the solution and shrinking set $V_A$}
We consider the following cut-off function $\chi_0\in C^{\infty}_0([0,+\infty))$ with $supp(\chi_0)\subset [0,2]$ and $\chi_0\equiv 1$ in $[0,1]$. Then, we define
$$\chi(y,s)=\chi_0\left(\frac{|y|}{K_0s^{1/2}}\right),$$ with $K_0$ large enough. We introduce
\begin{equation}\label{def q_e, q_b}
q_b=q\chi \text{ and } q_e=q(1-\chi),
\end{equation}
where $q$ is defined in \eqref{equality w and q}, which yields
$$q=q_b+q_e.$$
The differential operator $\mathcal{L}$ defined in \eqref{def op L} is a self-adjoint operator on its domain included in $L^2_\rho(\R^N)$, where $\rho(y)= \frac{e^{-\frac{|y|^2}{4}}}{(4\pi)^{\frac{N}{2}}}$. The spectrum of this operator consists of the following eigenvalues:
$$spec\ \mathcal{L}=\{1-\frac{n}{2}, n\in \N\}.$$
The corresponding eigenfunctions for $\beta=(\beta_1,...,\beta_N)\in \N^N$, with $\ |\beta|=n$, are
$$y=(y_1,...,y_N)\mapsto h_\beta=h_{\beta_1}(y_1)...h_{\beta_N}(y_N), \text{ for } |\beta|=\beta_1+...+\beta_N,$$
where $h_{\beta_i}$ are the (rescaled) Hermite polynomials given by
\begin{equation}\label{defhm}
h_{m}(\xi)=\underset{j=0}{\overset{[m/2]}{\sum}}\frac{m!}{j!(m-2j)!}(-1)^j\xi^{m-2j}, \text{ for } m\in \N \text{ and } \xi\in \R,
\end{equation}
where $[m/2]$ is the floor of $m/2$. Notice that these polynomials satisfy 
\begin{equation}\label{orthogonality property}
\int_{\R}h_{m}(\xi)h_{l}(\xi)\rho(\xi) d\xi=2^{m}m!\delta_{m,l}.
\end{equation}
Given $r\in L^2_\rho(\R^N)$, we define the component of $r$ on the eigenspace corresponding to the eigenvalue $\lambda=1-\frac{n}{2}$ by
\begin{equation}\label{def projection Pn}
P_n (r)(y)=\underset{|\beta|=n}{\sum} \pi_\beta (r)\ h_\beta(y),
\end{equation}
where 
\begin{equation}\label{def q_beta and k_beta}
\pi_\beta(r)=\int_{\R^N} k_\beta(y)r(y)\rho(y) dy \text{ and }  k_\beta(y)=h_\beta/\|h_\beta\|_{L^2_\rho}^2.
\end{equation}
Since the eigenfunctions of $\mathcal{L}$ span the whole space $L^2_\rho(\R^N)$, then we can write $q$ as
\begin{equation}\label{decomposition r}
q(y,s)=\underset{n=0}{\overset{2}{\sum}}P_n(q_b)(y,s)+q_{b,-}(y,s)+q_e=\underset{n=0}{\overset{2}{\sum}}\underset{|\beta|=n}{\sum}\pi_\beta(q_b)(s)h_\beta(y)+q_{b,-}(y,s)+q_e,
\end{equation}
where $q_{b,-}=P_-(q_b)=\underset{n\geq 3}{\sum}P_n(q_b)$. In the following proposition, we introduce the shrinking set $V_A$ which is crucial for the proof.

\begin{prp}[A set shrinking to zero]\label{def V_A}
For all $A\geq 1$ and $s\geq s_0$, we define $V_A(s)$ as the set of functions in $L^\infty(\R^N)$ such that, for all $r(s)\in V_A(s)$, we have, for all $y\in \R^\N$, the following:
\begin{equation}
\begin{split}
&|\pi_\beta(r_b)(s))|\leq As^{-2},\ |\beta|=0,1,\\
&|\pi_\beta(r_b)(s)|\leq A^2(\log s) s^{-2},\ |\beta|=2,\\
&|r_{b,-}(y,s)|\leq A(1+|y|^3)s^{-2},\\
&\|r_e(s)\|_{L^\infty(\R^N)}\leq A^2s^{-1/2},
\end{split}
\end{equation}
where $r_-$, $r_e$ and $r_n$ are defined as previously. Then, for all $s\geq s_0$ and $r(s)\in V_A(s),$, we have the following, for all $y\in \R^\N$,
\begin{enumerate}[label=(\roman*)]
\item for all $y\in \R^\N, \ |r_b(y,s)|\leq CA^2\frac{\log s}{s^2}(1+|y|^3)$,\label{i}
\item $\|r(s)\|_{L^\infty(\R^N)}\leq \frac{CA^2}{\sqrt{s}}$.\label{ii}
\end{enumerate}
\end{prp}

\begin{proof}
The proof follows directly from the definition of $V_A(s)$.
\end{proof}

It was proved in \cite{EZ2011}, the existence of such $q$ and $A>0$ such that $q(s)\in V_A(s)$ for all $s\geq s_0$. Moreover, it has been shown the stability of its initial data. Therefore, the set $S_{a,T,h}\neq \emptyset$. We are interested in blow-up solutions $u$ which blow up at $T$ at the blow-up point $a$ such that $q(s)\in V_A(s)$ for all $s\geq s_0=-\log T$ .


\subsection{Difference between two blow-up solutions}
Given $u_i\in S_{a_i,T_i,h_i}$ for some $T_i>0$ (we assume $T_2\geq T_1$ for simplicity), $a_i\in \R^N$ and $h_i$ verifies \eqref{condition h}, where $i=1,2$. Since from \eqref{behaviour1}, $w_1$ and $w_2$ have the same profile with error $\frac{1}{\sqrt{s}}$, then we have for free 

\begin{equation}\label{eq:rough}
\|g(s)\|_{W^{1,\infty}(\R^N)}=O\left(\frac{1}{\sqrt{s}}\right),
\end{equation}
where 
\begin{equation}\label{defw1w2g}
\begin{aligned}
&w_1(y,s)=(T_1-t)^{\frac{1}{p-1}}u_1(x,t),\\
&w_2(y,s)=(T_1-t)^{\frac{1}{p-1}}u_2(x-a_1+a_2,t-T_1+T_2),\\
&g(y,s)=w_1(y,s)-w_2(y,s).
\end{aligned}
\end{equation}
One may see, from \eqref{eq1:cv} and \eqref{defw1w2g}, that $g$ satisfies

\begin{equation}\label{eq3: g}
\partial_s g=\mathcal{L}g+\alpha_1 g+\alpha_2,
\end{equation}
where $\mathcal{L}$ is defined in \eqref{def op L} and
\begin{equation}\label{alpha1}
\alpha_1=\frac{|w_1|^{p-1}w_1-|w_2|^{p-1}w_2}{w_1-w_2}-\frac{p}{p-1},
\end{equation}

\begin{equation}\label{alpha2}
\alpha_2(y,s)=e^{-\frac{p}{p-1} s}(h_1(e^{\frac{1}{p-1} s}w_1,e^{\frac{p+1}{2(p-1)} s}\nabla w_1)-h_2(e^{\frac{1}{p-1} s}w_2,e^{\frac{p+1}{2(p-1)} s}\nabla w_1)),\ \forall s\geq s_0, y\in \R^N.
\end{equation}

\subsection{Strategy of the proof of Theorem \ref{Thm: ch1}'}
The following five sections are devoted to the proof of Theorem \ref{Thm: ch1}'. Let us briefly give the main ideas of how we proceed in each section.\\
\textbf{Section \ref{section l infinity estimates}:} We will first prove the boundedness of $w$ and $\nabla w$ in $L^\infty_{y,s}$. This will give us some parabolic regularities for $\alpha_1$ and $\alpha_2$.\\
\textbf{Section \ref{L 2 estimate}:} We prove an estimate of $g$ in $L^2_\rho$ by expanding $g$ on the eigenspaces of $\mathcal{L}$, which will yield to an estimate in $\{|y|\leq R\}$ for a $R>0$.\\
\textbf{Section \ref{L infinity estimate in blow-up region}:} We will extend  the previous estimate to the blow-up region $\{|y|\leq K_0\sqrt{s}\}$.\\
\textbf{Section \ref{No blow up threshold}:} We establish a new no blow-up under some threshold criterion for an equality system involving the solution and its gradient.\\
\textbf{Section \ref{Conclusion Thm 2}:} Near the blow-up time and point and in the set $\{|y|\geq K_0\sqrt{s}\}$, we will complete the estimate using techniques which are similar in the spirit of Giga and Kohn in \cite{GK1989} and Tayachi and Zaag in \cite{TZ2015}. Though, we will need genuine new parabolic regularity estimates to control the nonlinear gradient term using iterative arguments.

A drawing might help the reader to picture the strategy better, where we will also give further ideas about the behaviour of the solution in each region.

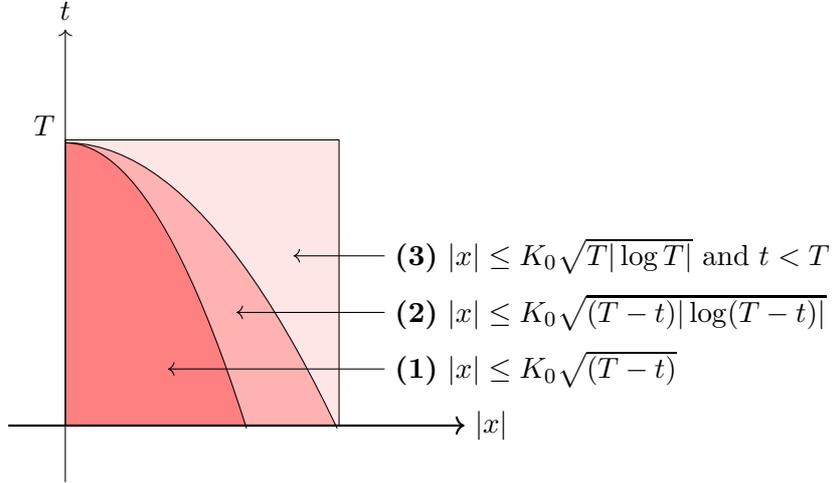
\begin{figure}[h!]
\center
\begin{tikzpicture}[scale=1.5]
    \fill [red!10](0,0) rectangle (2.4,2.54);
    \fill[red!30, scale=0.5, domain=0:4.7675, variable=\x] plot({\x},{-(\x/3)*2)*(\x/3)+5});
    \fill[red!30, scale=0.5] (0,0)--(4.7675,0)--(0,5) -- cycle;
    \fill[red!50, scale=0.5, domain=0:3.178, variable=\x] plot({\x},{-(\x/2)*2)*(\x/2)+5});
    \fill[red!50, scale=0.5] (0,0)--(3.178,0)--(0,5) -- cycle;
    \draw[thick,->] (-0.5,0) -- (3.5,0);
    \draw [right] (3.5,0) node{$|x|$};
    \draw[->] (0,-0.5) -- (0,3.5);
    \draw [above](0,3.5) node{$t$};
    \draw [left] (0,2.65) node{$T$};
    \draw (0,0) rectangle (2.4,2.525);
    \draw[scale=0.5, domain=0:4.768, variable=\x] plot({\x},{-(\x/3)*2)*(\x/3)+5});
    \draw[scale=0.5, domain=0:3.178, variable=\x] plot({\x},{-(\x/2)*2)*(\x/2)+5});
    \draw[<-] (2,1.5) -- (2.8,1.5);
    \draw [right] (2.8,1.5) node{\textbf{(3)} $|x|\leq K_0\sqrt{T|\log T|}$ and $t<T$};
    \draw[<-] (1.5,1) -- (2.8,1);
    \draw [right] (2.8,1) node{\textbf{(2)} $|x|\leq K_0\sqrt{(T-t)|\log(T-t)|}$};
    \draw[<-] (0.9,0.5) -- (2.8,0.5);
    \draw [right] (2.8,0.5) node{\textbf{(1)} $|x|\leq K_0\sqrt{(T-t)}$};
    
\end{tikzpicture}
\end{figure}
\textbf{Region (1):} We have $w\sim f(0)=\kappa$ given in \eqref{expression f} and \eqref{equality w and q}. Thus, $u\sim \kappa(T-t)^{-\frac{1}{p-1}}$ with $\kappa=(p-1)^{-\frac{1}{p-1}}$.

\textbf{Region (2):} We have $w\sim f(\frac{y}{s})$. Thus, $u\sim (T-t)^{-\frac{1}{p-1}}f\left(\frac{x}{\sqrt{(T-t)|\log(T-t)|}}\right)$.

\textbf{Region (3):} We have $u'\sim u^p$, (Small diffusion term). 

\section{$L^\infty_{y,s}$ estimates and parabolic regularity}\label{section l infinity estimates}

Let $u$ be a solution of \eqref{eq:1} such that $u\in S_{a,T,h}$. We will first prove the boundedness of $w$ and $\nabla w$, where $w$ is defined in \eqref{cv}.
\begin{lemma}[\textbf{$L^\infty_{y,s}$ bound on $w$}]\label{bound w}
There exists $M>0$ such that for all $y\in \R^N$ and $s\geq s_0$, $| w(y,s)|\leq M$.
\end{lemma}
\begin{proof}
Using \eqref{equality w and q}, we have for all $y\in \R^N$ and $s\geq s_0$,
$$|w(y,s)|\leq f\left(\frac{y}{\sqrt{s}}\right)+\frac{\kappa}{2ps}+|q(y,s)|.$$
Since $q_i(s)\in V_A(s)$ and $\kappa=\max f$, then from Proposition \ref{def V_A} $(ii)$, we have
$$|w(y,s)|\leq \kappa+\frac{\kappa}{2ps}+C\frac{A^2}{\sqrt{s}},$$
which concludes the proof.
\end{proof}

\begin{prp}[\textbf{Parabolic regularity}]\label{bound nabla w}
There exists $s_1\geq s_0$ such that for all $s\geq s_1$, we have $\|\nabla w(s)\|_{L^\infty(\R^N)}\leq M'$.
\end{prp}
The proof is a direct consequence of the following Lemma:

\begin{lemma}
There exist $s_1\geq s_0$ and $\delta_0>0$ such that for all $s\geq s_1$ and for all $s'\in[\overline{s},s]$ with $\overline{s}=s-\delta_0$ , we have $\|\nabla w(s')\|_{L^\infty(\R^N)}\leq \frac{\overline{M}}{\sqrt{s'-\overline{s}}}$ where $\overline{M}=2C(M+2(1+M^p)\sqrt{\delta_0}(1+C_1(\bar s))$, $C_1(s)=e^{-\frac{p-\gamma}{p-1}s}$ and $M$ is given in Lemma \ref{bound w}.
\end{lemma}
\begin{proof}

We assume $s\geq s_1$ with $s_1$ large enough and going to be determined afterwards. The solution of \eqref{eq1:cv} is given by
\begin{equation*}
w(s')=e^{(s'-s_0)\mathcal{L}_p}w(\overline{s})+\int^{s'}_{\overline{s}}e^{(s'-t)\mathcal{L}_p}F(t)dt,
\end{equation*}
where
\begin{equation*}
\begin{split}
&\mathcal{L}_p=\mathcal{L}-\frac{p}{p-1}=\Delta-\frac{y}{2}.\nabla-\frac{1}{p-1},\\
&F(t)=|w|^{p-1}(t)w(t)+e^{-\frac{p}{p-1} t}h(e^{\frac{1}{p-1} t}w(t),e^{\frac{p+1}{2(p-1)}t}\nabla w(t)).
\end{split}
\end{equation*}
Then, 
\begin{equation}\label{w bound}
\|\nabla w(s')\|_{L^\infty(\R^N)}\leq\|\nabla e^{(s'-\overline{s})\mathcal{L}_p}w(\overline{s})\|_{L^\infty(\R^N)}+\int^{s'}_{\overline{s}}\|\nabla e^{(s'-t)\mathcal{L}_p}F(t)\|_{L^\infty(\R^N)}dt.
\end{equation}
To simplify the notations, we will define, for all $s\geq s_0$,
\begin{equation}
C_0(s)=e^{-\frac{p+1}{2(p-1)}s}, C_1(s)=e^{-\frac{p-\gamma}{p-1}s}, C_2(s)=e^{-\frac{2p-\overline{\gamma}(p+1)}{p-1}s}.
\end{equation}
We can see from \cite{BK1994} that we have for all $\theta>0$,
\begin{equation*}
e^{\theta\mathcal{L}}(y, x) =\frac{e^\theta}{(4\pi(1 - e^{-\theta}))^\frac{N}{2}}\exp \left[-\frac{|ye^{-\theta/2} - x|^2}{4(1 - e^{-\theta})}\right].
\end{equation*}
Therefore, for $r\in W^{1,\infty}(\R^N)$, we obtain

\begin{equation}\label{inequality theta}
\|\nabla e^{\theta\mathcal{L}}r\|_{L^\infty(\R^N)}\leq \frac{Ce^{\frac{\theta}{2}}}{\sqrt{1-e^{-\theta}}}\|r\|_{L^\infty(\R^N)}.
\end{equation}
Consequently, we have from  \eqref{inequality theta} and Lemma \ref{bound w}, for all $s'\in [\overline{s},s]$,

\begin{equation}\label{linear part w conclusion}
\|\nabla e^{(s'-\overline{s})\mathcal{L}_p}w(\overline{s})\|_{L^\infty(\R^N)}
\leq \frac{CC_0(s'-\bar s)}{\sqrt{1-e^{-(s'-\overline{s})}}}\|w(\overline{s})\|_{L^\infty(\R^N)}
\leq \frac{CM}{\sqrt{s'-\overline{s}}}.
\end{equation}
We have, for the rest term, that
\begin{equation*}
\begin{split}
\int^{s'}_{\bar s}\|\nabla &e^{(s'-t)\mathcal{L}_p}F(t)\|_{L^\infty(\R^N)}dt\\
\leq C&\int^{s'}_{\bar s}\frac{C_0(s'-t)}{\sqrt{1-e^{-(s-t)}}}\left(\|w\|^p_{L^\infty(\R^N)}+C_1(t)\|w\|^\gamma_{L^\infty(\R^N)}+C_2(t)\|\nabla w\|^{\overline{\gamma}}_{L^\infty(\R^N)}+1\right)dt\\
\leq C&\Bigg((1+M^p)\int^{s'}_{\bar s}\frac{C_0(s'-t)}{\sqrt{1-e^{-(s'-t)}}}(1+C_1(t))dt\\
&\ \left.+\int^{s'}_{\bar s}\frac{C_0(s'-t)}{\sqrt{1-e^{-(s'-t)}}}C_2(t)\|\nabla w\|^{\overline{\gamma}}_{L^\infty(\R^N)}dt\right).
\end{split}
\end{equation*}
Then,
\begin{equation}\label{non linear conclusion}
\begin{split}
\int^{s'}_{\overline{s}}\|\nabla e^{(s'-t)\mathcal{L}_p}F(t)\|_{L^\infty(\R^N)}dt \leq C\Bigg(&2(1+M^p)\sqrt{\delta_0}(1+C_1(\bar s))\\
&\left.+C_2(\overline{s})\int^{s'}_{\overline{s}}\frac{1}{\sqrt{s'-t}}\|\nabla w(t)\|^{\overline{\gamma}}_{L^\infty(\R^N)}dt\right).
\end{split}
\end{equation}
From \eqref{w bound},\eqref{linear part w conclusion} and \eqref{non linear conclusion}, there exists a constant $C>0$ such that
\begin{equation}\label{controll 2}
\begin{split}
\|\nabla w(s')\|_{L^\infty(\R^N)}&\leq C\left(\frac{M}{\sqrt{s'-\overline{s}}}+2(1+M^p)\sqrt{\delta_0}(1+C_1(\bar s))\right.\\
&\hspace{1.2cm} \left.+C_2(\overline{s})\int^{s'}_{\overline{s}}\frac{1}{\sqrt{s'-t}}\|\nabla w(t)\|^{\overline{\gamma}}_{L^\infty(\R^N)}dt\right)\\
&\leq \frac{\overline{M}}{2\sqrt{s'-\overline{s}}}+CC_2(\overline{s})\int^{s'}_{\overline{s}}\frac{1}{\sqrt{s'-t}}\|\nabla w(t)\|^{\overline{\gamma}}_{L^\infty(\R^N)}dt,
\end{split}
\end{equation}
with $\overline{M}=2C(M+2(1+M^p)\sqrt{\delta_0}(1+C_1(\bar s))$ and $C\geq \sqrt{\delta_0}$. We want to show for all $s'\in[\overline{s},s]$
\begin{equation*}
\|\nabla w(s')\|_{L^\infty(\R^N)}\leq \frac{\overline{M}}{\sqrt{s'-\overline{s}}}.
\end{equation*}
Provided that $s_1$ is large enough, we proceed again by contradiction, i.e. we assume there exists $s^*=\inf\{s'\geq s_1, \|\nabla w(s')\|_{L^\infty(\R^N)}>\frac{\overline{M}}{\sqrt{s'-\overline{s}}}\}$.  Then,
\begin{equation}\label{absurde hyp}
\|\nabla w(s^*)\|_{L^\infty(\R^N)}\geq \frac{\overline{M}}{\sqrt{s^*-\overline{s}}},
\end{equation}
and for all $s'\in [s_0,s^*]$,
\begin{equation*}
\|\nabla w(s')\|_{L^\infty(\R^N)}\leq \frac{\overline{M}}{\sqrt{s^*-\overline{s}}}.
\end{equation*}
By continuity, we get
\begin{equation}\label{contradition 2}
\|\nabla w(s^*)\|_{L^\infty(\R^N)}= \frac{\overline{M}}{\sqrt{s^*-\overline{s}}}.
\end{equation}
Then, from \eqref{controll 2},
\begin{equation}\label{lemma s^*}
\|\nabla w(s^*)\|_{L^\infty(\R^N)}\leq \frac{\overline{M}}{2\sqrt{s^*-\overline{s}}}+CC_2(\overline{s})\overline{M}^{\overline{\gamma}}\int^{s^*}_{\overline{s}}\frac{1}{\sqrt{s^*-t}}\frac{1}{(t-\overline{s})^{\frac{\overline{\gamma}}{2}} }dt,
\end{equation}
If we denote $\theta=t-\overline{s}$ and $\theta^*=s^*-\overline{s}$, then we have 
$$\int^{s^*}_{\overline{s}}\frac{1}{\sqrt{s^*-t}}\frac{1}{(t-\overline{s})^{\frac{\overline{\gamma}}{2}} }dt=\int^{\theta^*}_0(\theta^*-\theta)^{-\frac{1}{2}}\theta^{-\frac{\overline{\gamma}}{2}}\ dt= \left(\int^{\frac{\theta}{2}^*}_0+\int^{\theta^*}_{\frac{\theta}{2}^*}\right)(\theta^*-\theta)^{-\frac{1}{2}}\theta^{-\frac{\overline{\gamma}}{2}}\ \theta.$$
For the first integral, since $0\leq\overline{\gamma}<\frac{2p}{p+1}<2$, we have 
$$\int^{\frac{\theta}{2}^*}_0 (\theta^*-\theta)^{-\frac{1}{2}}\theta^{-\frac{\overline{\gamma}}{2}}\ dt\leq \left(\frac{\theta^*}{2}\right)^{-\frac{1}{2}}\int^{\frac{\theta}{2}^*}_0 \theta^{-\frac{\overline{\gamma}}{2}}\ dt=\frac{2}{2-\overline{\gamma}}\left(\frac{\theta^*}{2}\right)^{\frac{1-\overline{\gamma}}{2}}=\frac{2}{2-\overline{\gamma}}\left(\frac{\delta_0}{2}\right)^{\frac{1-\overline{\gamma}}{2}}\left(\frac{\theta^*}{\delta_0}\right)^{\frac{1-\overline{\gamma}}{2}}.$$
In addition, since $0\leq \frac{\theta^*}{\delta_0}\leq 1$, it follows that
$$\int^{\frac{\theta}{2}^*}_0 (\theta^*-\theta)^{-\frac{1}{2}}\theta^{-\frac{\overline{\gamma}}{2}}\ dt\leq \frac{2}{2-\overline{\gamma}}\left(\frac{\delta_0}{2}\right)^{\frac{1-\overline{\gamma}}{2}}\left(\frac{\theta^*}{\delta_0}\right)^{-\frac{1}{2}}=\frac{2^{\frac{1+\overline{\gamma}}{2}}\delta_0^{1-\frac{\overline{\gamma}}{2}}}{2-\overline{\gamma}}(\theta^*)^{-\frac{1}{2}}.$$
For the second one, we get 
\begin{equation*}
\begin{split}
\int_{\frac{\theta}{2}^*}^{\theta^*} (\theta^*-\theta)^{-\frac{1}{2}}\theta^{-\frac{\overline{\gamma}}{2}}\ dt &\leq \left(\frac{\theta^*}{2}\right)^{-\frac{\overline{\gamma}}{2}}\int_{\frac{\theta}{2}^*}^{\theta^*} (\theta^*-\theta)^{-\frac{1}{2}}\ dt=2\left(\frac{\theta^*}{2}\right)^{\frac{1-\overline{\gamma}}{2}}=2^{\frac{1+\overline{\gamma}}{2}}\delta_0^{\frac{1-\overline{\gamma}}{2}}\left(\frac{\theta^*}{\delta_0}\right)^{\frac{1-\overline{\gamma}}{2}}\\
&\leq 2^{\frac{1+\overline{\gamma}}{2}}\delta_0^{1-\frac{\overline{\gamma}}{2}}(\theta^*)^{-\frac{1}{2}}.
\end{split}
\end{equation*}
Therefore, \eqref{lemma s^*}) becomes
\begin{equation*}
\|\nabla w(s^*)\|_{L^\infty(\R^N)}\leq \frac{\overline{M}}{2\sqrt{s^*-\overline{s}}}+Ce^{-\frac{2p-\overline{\gamma}(p+1)}{p-1}(s-\delta_0)}\overline{M}^{\overline{\gamma}}2^{\frac{1+\overline{\gamma}}{2}}\delta_0^{1-\frac{\overline{\gamma}}{2}}\left(1+\frac{1}{2-\overline{\gamma}}\right)\frac{1}{\sqrt{s^*-\overline{s}}}.
\end{equation*}
We just need to take $s_1>-\frac{p-1}{2p-\overline{\gamma}(p+1)}\log \left(\frac{\overline{M}^{1-\overline{\gamma}}}{C^*2^{\frac{3+\overline{\gamma}}{2}}\delta_0^{1-\frac{\overline{\gamma}}{2}}(1+\frac{1}{2-\overline{\gamma}})}\right)+\delta_0$ in order to have $\|\nabla w(s^*)\|_{L^\infty(\R^N)}< \frac{\overline{M}}{\sqrt{s^*-\overline{s}}}$, which contradicts \eqref{contradition 2}. Thus, it concludes the proof.
\end{proof}

Now, we can derive the following estimates of $\alpha_1$ and $\alpha_2$.
\begin{lemma}[\textbf{Bounds on }$\mathbf{\alpha_1}$ \textbf{ and }$\mathbf{\alpha_2}$]\label{lemma:estimation alpha_1/2}
There exists a constant $C>0$, such that for all $y\in \R^N$ and $s\geq s_1$, we have

\begin{equation}\label{estimation alpha_1 1}
|\alpha_1(y,s)|\leq C \left(\frac{1+|y|^2}{s}\right),
\end{equation}
\begin{equation}\label{alpha_1 sqrt s}
\alpha_1(y,s)\leq \frac{C}{\sqrt{s}},
\end{equation}
\begin{equation}\label{estimation alpha_1 3}
\left|\alpha_{1}(y,s)-\frac{1}{2s}\left(N-\frac{|y|^2}{2}\right)\right|\leq C\left(\frac{1+|y|^4}{s^{3/2}}\right),
\end{equation}
\begin{equation}\label{estimation alpha_2}
|\alpha_2(y,s)|\leq Ce^{ -\nu s},
\end{equation}
where $\alpha_1$ and $\alpha_2$ are defined in \eqref{alpha1}, \eqref{alpha2}, and $\nu$ is defined in \eqref{estimation dim 1}.
\end{lemma}
\begin{proof}

We first claim the following: 

\begin{claim}\label{claim alpha 1}
\begin{enumerate}[label=(\roman*)]
There exists $w_0\in (w_1,w_2)$, such that
\item $|\alpha_1|\leq C(M)|w_0-\kappa|$,
\item $|\alpha_1-p(p-1)\kappa^{p-2}(w_0-\kappa)|\leq C(w_0-\kappa)^2$,\label{Claim alpha1 second order}
\end{enumerate}
where $M$ is given in Lemma \ref{bound w}.
\end{claim}
Let us finish the proof and we will come back later to the claim.
\begin{itemize}
\item \textbf{Proof of \eqref{estimation alpha_1 1}: }We know from \eqref{expression f} and \eqref{equality w and q}, for $i=1,2$ and for all $y\in \R^N$ and $s\geq s_0$, that 
$$|w_i(y,s)-\kappa|\leq \left|f\left(\frac{y}{\sqrt{s}}\right)-f(0)\right|+|q_i(y,s)|+\frac{N\kappa}{2ps}.$$
Using a Taylor expansion, we get
$$|w_i(y,s)-\kappa|\leq C\frac{|y|^2}{s}+|q_i(y,s)|+\frac{N\kappa}{2ps}.$$
From \eqref{def q_e, q_b}, we write again $q$ as
$$q_i=q_{i,e}+q_{i,b}.$$
Using Proposition \ref{def V_A}, we get
\begin{equation}\label{q_bi}
\begin{split}
|q_{i,e}(y,s)|&\leq C\left(\frac{A}{s^2}(1+|y|)+\frac{A^2\log s}{s^2}(1+|y|^2)\right.\\
&\hspace{1.1cm}\left.+\frac{A}{s^2}(1+|y|^3)\right)1_{\{|y|\leq 2K_0\sqrt{s}\}}(y, s).
\end{split}
\end{equation}
Then,
\begin{equation*}
\begin{split}
|q_{i,b}(y,s)|&\leq C \left(\frac{A}{s^2}(1+|y|)+\frac{A^2\log s}{s^2}(1+|y|^2)\right.\\
&\hspace{1.1cm}\left.+\frac{A}{s^2}(1+|y|^2)(1+2K_0\sqrt{s})\right)1_{\{|y|\leq 2K_0\sqrt{s}\}}(y, s)\\
&\leq C \frac{A}{s^{3/2}}(1+|y|^2).
\end{split}
\end{equation*}

Again from Proposition \ref{def V_A}, we obtain
\begin{equation}\label{bound q_e}
\|q_{i,e}(s)\|_{L^\infty(\R^N)}\leq \frac{A^2}{\sqrt{s}}.
\end{equation}
Since we have $supp\ q_{e,i}=\{(y,s), |y|\geq 2K_0\sqrt{s}\}$, then 
\begin{equation}\label{q_ei}
|q_{i,e}(y,s)|\leq \frac{(Ay)^2}{4K_0^2s^\frac{3}{2}}.
\end{equation}
Then, 
\begin{equation}\label{control q_i}
|q_i(y,s)|\leq CA^2\frac{1+|y|^2}{s^{3/2}}.
\end{equation}

Since we have $|w_0-\kappa|\leq \max(|w_1-\kappa|,|w_2-\kappa|)$ from Lemma \ref{bound w}, therefore,
\begin{equation}\label{bound w-kappa}
|w_0-\kappa|\leq C \frac{1+|y|^2}{s}.
\end{equation}

Together with Claim $(i)$, we obtain \eqref{estimation alpha_1 1}.


\item \textbf{Proof of \eqref{alpha_1 sqrt s}: }Using \eqref{equality w and q} and Proposition \ref{def V_A} $(ii)$, we have, for $i=1,2$ and for all $s\geq s_0,\ y\in \R^N,$
$$w_i-\kappa\leq \frac{C}{s}+|q_i(y,s)|\leq \frac{C}{\sqrt{s}}.$$
Since $w_0-\kappa\in (w_1-\kappa,w_2-\kappa)$, with Claim $(i)$ and Proposition \ref{def V_A} $(ii)$, we may conclude \eqref{alpha_1 sqrt s}.

\item \textbf{Proof of \eqref{estimation alpha_1 3}: } From Claim $(ii)$, we have
\begin{equation}\label{DL alpha 1}
\alpha_1= p(p-1)\kappa^{p-2}|w_0-\kappa|+ O(|w_0-\kappa|^2).
\end{equation}
With a Taylor expansion in the variable $z=\frac{y}{\sqrt{s}}$, we have
$$f(z)=\kappa\left((p-1)+\frac{(p-1)^2}{4p}|z|^2\right)^{-\frac{1}{p-1}}=\kappa\left(1-\frac{1}{4p}|z|^2+o(|z|^4)\right).$$
Therefore, with \eqref{equality w and q} and \eqref{control q_i},
\begin{equation}
\begin{split}
w_0(y,s)-\kappa&=f\left(\frac{y}{\sqrt{s}}\right)-\kappa+q_0(y,s)+\frac{N\kappa}{2ps}\\
&=-\frac{\kappa }{4ps}|y|^2+O\left(\frac{|y|^4}{s^2}\right)+O\left(\frac{1+|y|^2}{s^{\frac{3}{2}}}\right)+\frac{N\kappa}{2ps}\\
&=\frac{\kappa}{2ps}\left(N-\frac{|y|^2}{2}\right)+O\left(\frac{|y|^4}{s^2}\right)+O\left(\frac{1+|y|^2}{s^{\frac{3}{2}}}\right).
\end{split}
\end{equation}
Together with \eqref{DL alpha 1} and \eqref{bound w-kappa}, we get
\begin{equation}
\begin{split}
\alpha_1(y,s)=\frac{1}{2s}\left(N-\frac{|y|^2}{2}\right)+O\left(\frac{|y|^4}{s^2}\right)+O\left(\frac{1+|y|^2}{s^{\frac{3}{2}}}\right)+O\left(\frac{1+|y|^4}{s^{2}}\right).
\end{split}
\end{equation}
Hence, \eqref{estimation alpha_1 3}.

\item \textbf{Proof of \eqref{estimation alpha_2}: } Now, for the last estimation on $\alpha_2$, by definition \eqref{alpha2} of $\alpha_2$,  \eqref{condition h}, Proposition \eqref{bound nabla w} and Lemma \eqref{bound w}, we obtain
\begin{equation}\label{alpha 2 bound computations}
\begin{split}
|\alpha_2(s)|&\leq C\left(e^{-\frac{p-\gamma}{p-1}s}(|w_1(s)|^\gamma+|w_2(s)|^\gamma)+e^{-\frac{2p-\overline{\gamma}(p+1)}{p-1}s}(|\nabla w_1(s)|^{\overline{\gamma}}+|\nabla w_2(s)|^{\overline{\gamma}})+e^{-\frac{p}{p-1} s}\right)\\
&\leq Ce^{-\nu s}\left(2M^\gamma+2M'^{\overline{\gamma}}+1\right)\\
&\leq Ce^{-\nu s},
\end{split}
\end{equation}
where $\frac{p-\gamma_0}{p-1}=\nu$ and $\gamma_0=\max\{\gamma,\overline{\gamma}(p+1)-p\}$.
\end{itemize}
This concludes the proof of Lemma \ref{lemma:estimation alpha_1/2} assuming Claim \ref{claim alpha 1} holds.
It remains to prove the claims we asserted at the beginning of the proof.
\end{proof}

\begin{proof}[Proof of Claim \ref{claim alpha 1}]
We know by definition \eqref{alpha1}, for some $w_0\in (w_1,w_2)$, that 
\begin{equation}\label{expression alpha_1}
\alpha_1= p|w_0|^{p-1}-\frac{p}{p-1}= p(|w_0|^{p-1}-\kappa^{p-1}).
\end{equation}
\begin{enumerate}[label=(\roman*)]
\item  If $|w_0|\geq \frac{\kappa}{2}$, then$|w_0|\leq \max(|w_1|,|w_2|)\leq M$, from Lemma \ref{bound w}. Therefore, $x\mapsto|x|^{p-1}$ is $C^1$ on the compact set $\{\frac{\kappa}{2}\leq |x|\leq \max(M,\kappa)\}$. Thus, \eqref{expression alpha_1} becomes
$$|\alpha_1|\leq C(M,p)|w_0-\kappa|.$$\\
If $|w_0|\leq \frac{\kappa}{2}$, then we have that $|w_0-\kappa|\geq \frac{\kappa}{2}$. Since we have that $|w_0|\leq M$, from Lemma \ref{bound w}, then
$$|\alpha_1|\leq pM^{p-1}\leq C(M,p)|w_0-\kappa|.$$
\item If $|w_0|\geq\frac{\kappa}{2}$, then we have $|w_0|\leq \max(|w_1|,|w_2|)\leq M$, from Lemma \ref{bound w}. Using a Taylor expansion on $\alpha_1$, we have
$$\alpha_1=p(p-1)\kappa^{p-2}(w_0-\kappa)+o((w_0-\kappa)^2),$$
which implies 
$$|\alpha_1-p(p-1)\kappa^{p-2}(w_0-\kappa)|\leq C|w_0-\kappa|^2.$$
If $|w_0|<\frac{\kappa}{2}$, then, with the same arguments used in $(i)$, we can conclude that \ref{Claim alpha1 second order} holds. This concludes the proof of Lemma \ref{lemma:estimation alpha_1/2}.
\end{enumerate}
\end{proof}

\section{$L^2_\rho$ estimate on $g$}\label{L 2 estimate}
Let $u_i$ for $i=1,2$ be solutions of \eqref{eq:1} with the perturbation term $h_i$ such that $u_i\in S_{a_i,T_i,h_i}$. Then, we have the following $L^2_\rho$ estimate of the difference between $w_i$ defined in \eqref{cv}:
\begin{prp}[$L^2_\rho$ estimate on $g$]\label{prp:g in L^2_rho}
We have
\begin{equation}\label{estimation L^2_rho}
\|g(s)\|_{L^2_\rho}=O\left(\frac{1}{s^2}\right),
\end{equation}
where $g$ is defined in \eqref{defw1w2g}. Moreover, for N=1, there exists $\sigma_0\in \R$ such that
\begin{equation}\label{estimation L^2_rho N=1}
\|w_1(s)-w_2(s+\sigma_0)\|_{L^2_\rho}= O\left( \max (\frac{e^{-s/2}}{s^3}, e^{-\nu s+C \sqrt{s}})\right),
\end{equation}
where $\nu$ is defined in \eqref{estimation dim 1}. We recall that $\rho(y)= \frac{e^{-\frac{|y|^2}{4}}}{(4\pi)^{\frac{N}{2}}}$.
\end{prp}

We will use the decomposition of $g$ according to the spectrum of $\mathcal{L}$ on $L^2_\rho(\R^N)$.  We have the following decomposition of $g$:
\begin{equation}\label{decomposition g}
g(y,s)=g_0(s)+g_1(s).y+ y^T g_2(s)y-2\ tr g_2(s)+ g_-(y,s),
\end{equation}
where 
\begin{equation*}
\begin{aligned}
g_0(s)&=P_0g(y,s),\\
g_1(s).y&=P_1g(y,s),\\
y^T g_2(s)y-2\ tr g_2(s)&=P_2g(y,s),\\
g_-(y,s):&=\underset{n\geq 3}{\sum}P_n g(y,s),
\end{aligned}
\end{equation*}
and $P_n$ is given in \eqref{def projection Pn}. Since $1-\frac{n}{2}$, for $n\geq 3$, are negative eigenvalues, it makes sense to take care of the non-negative eigenvalues, i.e. $n=0$, $n=1$ and $n=2$ in \eqref{decomposition g}.\\
From \eqref{orthogonality property} , we can notice that 
\begin{equation}\label{decomposition}
I(s)^2:=\|g(s)\|_{L^2_\rho}^2=\underset{n\geq 0}{\sum}l_n(s)^2.
\end{equation}
with $l_n(s)=\|P_n g(s)\|_{L^2_\rho}.$

\bigskip

In the following, we give the dynamics of $I(s)$ and $l_n(s)$:

\subsection{On the existence of a dominating component}
In this section,  we proceed in two steps:\\
\textbf{Step 1:} We project the equation \eqref{eq3: g} on the various components.\\
\textbf{Step 2:} We discuss the question of the existence of a dominating component of $g$.

\bigskip

\textbf{Step 1 (Projection on the components):} The following Lemma gives us the evolution of $I$ and $l_n$.
\begin{lemma}[\textbf{Evolution of $I(s)$ and $l_n(s)$}]\label{lemma:control of L_n with I nouveau}

There exist $s_2\geq s_1$ such that for all $n\in\N$ there exists $C_n$, such that for all $s\geq s_2$, we have
\begin{equation}\label{eq ln I}
|l_{n}'(s)+(\frac{n}{2}-1)l_{n}(s)|\leq C_n(\frac{1}{s}I(s)+e^{-\nu s}),
\end{equation}
and 
\begin{equation}\label{eq I}
I'(s)\leq (1-\frac{n+1}{2}+\frac{C}{\sqrt{s}})I(s)+Ce^{-\nu s}+\underset{k=0}{\overset{n}{\sum}}\frac{n+1-k}{2}l_{k}(s),
\end{equation}
where $\nu$ is defined in \eqref{estimation dim 1}.
\end{lemma}

Before proving this lemma, let us first give the following corollary related to 
$$\overset{\sim}{I}(s):=I(s)+se^{-\nu s}.$$
\begin{cor}\label{cor: dyn I and l tilde}
It holds that for all $n\in\N$, there exists $C_n$ such that for all $s\geq s_2$, we have
\begin{equation}\label{dyn I & l}
|l_{n}(s)'+(\frac{n}{2}-1)l_{n}(s)|\leq C_n\left(\frac{\overset{\sim}{I}(s)}{s}\right).
\end{equation}
Moreover, for $n=0,1$,
\begin{equation}\label{dyn I & l 0&1}
\overset{\sim}{I'}(s)\leq \left(1-\frac{n+1}{2}+\frac{C}{\sqrt{s}}\right)\overset{\sim}{I}(s)+\underset{k=0}{\overset{n}{\sum}}\frac{n+1-k}{2}l_{k}(s).
\end{equation}
In addition, for $n\geq 2$, there exists $s_1\geq s_0$ such that for all $s\geq s_1$,

\begin{equation}\label{dyn I & l n>1}
\overset{\sim}{I'}(s)\leq -\min\left(\frac{n+1}{2}-1,\nu\right)\overset{\sim}{I}(s)+\frac{C}{\sqrt{s}}\overset{\sim}{I}(s)+\underset{k=0}{\overset{n}{\sum}}\frac{n+1-k}{2}l_{k}(s),
\end{equation}
where $\nu$ is given in \eqref{estimation dim 1}.
\end{cor}

\begin{proof}
One may see from the definition of $\overset{\sim}{I}$ that it is easy obtain \eqref{dyn I & l}. For \eqref{dyn I & l 0&1} and \eqref{dyn I & l n>1}. We just need to notice that $(se^{-\nu s})'=(-\nu+\frac{1}{s})se^{-\nu s}$.
\end{proof}

Let us now give the proof of Lemma \ref{lemma:control of L_n with I nouveau}.
\begin{proof}[Proof of Lemma \ref{lemma:control of L_n with I nouveau}]
From \eqref{eq3: g}, we have, for $n\in \N, \beta\in \N^N$ with $|\beta|=n$,
\begin{equation}\label{eq: proof lemma}
g'_\beta(s)=\left(1-\frac{n}{2}\right)g_\beta(s)+\int\alpha_1(y,s)g(y,s)k_\beta(y)\rho(y) dy+\int\alpha_2(y,s)k_\beta(y)\rho(y) dy,
\end{equation}
for all $s\geq s_1$. Using Cauchy-Schwarz inequality and Lemma \ref{lemma:estimation alpha_1/2}, we get, for all $s\geq s_1$,
$$\left|g'_{\beta}(s)-\left(1-\frac{n}{2}\right)g_{\beta}(s)\right|
\leq C_\beta\left(\frac{1}{s}I(s)+e^{-\nu s}\right).$$
Using the definition of $l_n$ in \eqref{decomposition}, we obtain \eqref{eq ln I}.

For the second inequality, using Lemma \ref{lemma:estimation alpha_1/2} and Cauchy-Schwarz inequality, we have the following straightforward computations:
\begin{equation*}
\begin{aligned}
I(s)I'(s)&=(\mathcal{L}g(s)|g(s))_{L^2_\rho}+(\alpha_1(s)g(s)|g(s))_{L^2_\rho}+(\alpha_2(s)|g(s))_{L^2_\rho}\\
&\leq \left(1-\frac{n+1}{2}\right)I(s)^2+\left(\left(\mathcal{L}-1+\frac{n+1}{2}\right)g(s)|g(s)\right)_{L^2_\rho}+\frac{C}{\sqrt{s}}I(s)^2+Ce^{-\nu s}\|\rho\|_{L^1}^{\frac{1}{2}}I(s)\\
&\leq \left(1-\frac{n+1}{2}\right)I(s)^2+\underset{k\geq 0}{\sum}\left(\frac{n+1-k}{2}P_kg(s)|P_kg(s)\right)_{L^2_\rho}+\frac{C}{\sqrt{s}}I(s)^2+Ce^{-\nu s}I(s)\\
&\leq \left(1-\frac{n+1}{2}\right)I(s)^2+\underset{k\geq 0}{\sum}\frac{n+1-k}{2}l_k^2+\frac{C}{\sqrt{s}}I(s)^2+Ce^{-\nu s}I(s)\\
&\leq \left(1-\frac{n+1}{2}\right)I(s)^2+I(s)\underset{k\geq 0}{\sum}\frac{n+1-k}{2}l_k+\frac{C}{\sqrt{s}}I(s)^2+Ce^{-\nu s}I(s).
\end{aligned}
\end{equation*}
Therefore,
$$I(s)I'(s)\leq\left(1-\frac{n+1}{2}+\frac{C}{\sqrt{s}}\right)I(s)^2+Ce^{-\nu s}I(s)+I(s)\underset{k=0}{\overset{n}{\sum}}\frac{n+1-k}{2}l_{k}(s),$$
which gives us the second inequality. 
\begin{remark}
One may see from \eqref{decomposition} that $I$ is not differentiable if $I(s)=0$ for some $s\geq s_1$. We still need to justify the inequality \eqref{eq I} for this case. One may also see in this case that $I$ has a local minimum in $s$. Implying $I'(s)=0$, which gives us the inequality \eqref{eq I}.
\end{remark}

This concludes the proof of Lemma \ref{lemma:control of L_n with I nouveau}.
\end{proof}

\textbf{Step 2 (On the dominating component): }
In this step, we prove the following Proposition:
\begin{prp}[On the existence of a dominating component]\label{prp:switch}
$ $
\begin{enumerate}[label=(\roman*)]
\item For $i\in \{0,1\},\ l_i=O\left(\frac{\overset{\sim}{I}(s)}{s}\right)$.\label{l0,l1}

\item Introducing $n_0\in \N$ such that 
\begin{equation}\label{def n0}
n_0:=\max\{n\in \N, \frac{n+1}{2}-1\leq\nu\},
\end{equation}
where $\nu$ is given in \eqref{estimation dim 1}, we have
\begin{itemize}
\item either there exists $n\leq n_0,\ n\notin \{0,1\}$ so that $\overset{\sim}{I}(s)\sim l_n(s)$,
\begin{equation}\label{inequality m & n}
\forall m\neq n,\ l_m(s)=O\left(\frac{l_n(s)}{s}\right),
\end{equation}
and
\begin{equation}\label{inequality 2}
\forall s\geq \sigma,\ \frac{s^{-C}}{C'}\exp\left(\left(1-\frac{n}{2}\right)s\right)\leq \overset{\sim}{I}(s)\leq C's^C\exp\left(\left(1-\frac{n}{2}\right)s\right),
\end{equation}
for some $\sigma,C,C'>0$.

\item or there exist $\sigma, C>0$ such that
\begin{equation}\label{inequality 1}
\forall s\geq \sigma,\ \overset{\sim}{I}(s)\leq C \exp(-\nu s+C\sqrt{s}).
\end{equation}

\end{itemize}
\end{enumerate}
\end{prp}

Before embarking in the proof, we will first need the following Lemma:

\begin{lemma}\label{lemma:x,y}
Let be $s_*\in \R$. If for all $s\geq s_*$, $x(s)$ and $y(s)$ satisfy
\begin{equation}\label{ineq x,y}
0\leq x(s)\leq y(s) \text{ and } y(s)\underset{s\mapsto +\infty}{\longrightarrow} 0,
\end{equation}

with 
\begin{equation}\label{ineq x',y'}
x'(s)\geq -\frac{C}{s}y(s)\text{ and } y'(s)\leq -\frac{1}{2}y(s)+\frac{C}{\sqrt{s}}y(s)+\frac{1}{2}x(s),
\end{equation}
then either $x\sim y$ as $s\rightarrow +\infty$ or there exists a constant $C$ such that
$$\forall s\geq s_*,\ x(s)\leq \frac{C}{s}y(s).$$
\end{lemma}

\begin{proof}
Either $x\sim y$ as $s\rightarrow +\infty$ or 
\begin{equation}
\exists \varepsilon>0\text{ and }(s_n)_n\subset \R, s_n\rightarrow +\infty\text{ as }n\rightarrow +\infty \text{ s.t } x(s_n)<(1-\varepsilon) y(s_n).
\end{equation}
Our first claim is 
\begin{equation}\label{sequence}
\exists n_0\in \N, \forall s>s_{n_0}, x(s)<(1-\varepsilon)y(s).
\end{equation}
If our claim is false, then we can find $\sigma_1$ and $\sigma_2$ as large as we want such that 
$$x(\sigma_2)=(1-\varepsilon)y(\sigma_2) \text{ and } \forall \sigma\in [\sigma_1,\sigma_2[,x(\sigma)>(1-\varepsilon)y(\sigma).$$
Then, necessarily $x'(\sigma_2)\leq (1-\varepsilon)y'(\sigma_2).$ But 
\begin{equation}\label{ineq x' y'}
\begin{split}
(x'-(1-\varepsilon)y')(\sigma_2)&\geq -\frac{C}{\sigma_2}y(\sigma_2)+(1-\varepsilon)\left(\frac{1}{2}y(\sigma_2)-\frac{C}{\sqrt{\sigma_2}}y(\sigma_2)-\frac{1}{2}x(\sigma_2)\right)\\
&\geq -\frac{C}{\sigma_2}y(\sigma_2)+(1-\varepsilon)\left(\frac{1}{2}y(\sigma_2)-\frac{C}{\sqrt{\sigma_2}}y(\sigma_2)-\frac{1}{2}(1-\varepsilon)y(\sigma_2)\right)\\
&\geq \frac{\varepsilon}{2}(1-\varepsilon) y(\sigma_2)-\frac{C}{\sigma_2}y(\sigma_2)-(1-\varepsilon)\frac{C}{\sqrt{\sigma_2}}y(\sigma_2).
\end{split}
\end{equation}
We know that $y(\sigma_2)\neq 0$ because if our statement was false, then, from \eqref{ineq x,y} and \eqref{ineq x',y'}, we have
$$y(\sigma_2)=x(\sigma_2)=0,\ x'(\sigma_2)\geq 0,\ y(\sigma_2)\leq 0,$$
which implies, for $\sigma_2$ large enough, that
$$x(s)\geq y(s),\ \forall s\geq \sigma_2.$$
Together with \eqref{ineq x',y'}, we have $y'(s)\leq \frac{C}{\sqrt{s}}y(s)\leq C'y(s)$, for $\sigma_2$ large enough. Therefore, we get 
$$y(s)\leq y(\sigma_2)e^{C'(s-\sigma_2)}=0.$$
Thus,
$$y(s)= 0 \text{ and } x(s)=0,\ \forall s\geq \sigma_2,$$
which contradicts \eqref{sequence}.	Therefore, from \eqref{ineq x' y'}, if $\sigma_2$ is large enough, then $x'(\sigma_2)>(1-\varepsilon)y'(\sigma_2)$. Hence, a contradiction and our claim is true.\\
From the claim, we can define $z(s)=\frac{x(s)}{y(s)}<1-\varepsilon$, for $s>s_{n_0}.$ Therefore, from the inequalities \eqref{ineq x',y'}, we have
$$z'(s)\geq -\frac{C}{s}+\frac{1}{2}z(s)-\frac{C}{\sqrt{s}}z(s)-\frac{1}{2}z(s)^2\geq -\frac{C}{s}-\frac{C}{\sqrt{s}}z(s)+\frac{\varepsilon}{2}z(s).$$
Hence, $\frac{d}{ds}(e^{2C\sqrt{s}-\frac{\varepsilon}{2}s}z(s))\geq -\frac{C}{s}e^{2C\sqrt{s}-\frac{\varepsilon}{2}s}$. This gives
$$z(s)\leq Ce^{-2C\sqrt{s}+\frac{\varepsilon}{2}s}\int_s^{+\infty}e^{2C\sqrt{\sigma}-\frac{\varepsilon}{2}\sigma}\sigma^{-1} d\sigma\leq Ce^{-2C\sqrt{s}+\frac{\varepsilon}{2}s}s^{-1}\int_s^{+\infty}e^{2C\sqrt{\sigma}-\frac{\varepsilon}{2}\sigma} d\sigma.$$
Since, we have
\begin{equation}
\begin{split}
\gamma(s)&:=\int_s^{+\infty}e^{2C\sqrt{\sigma}-\frac{\varepsilon}{2}\sigma} d\sigma\\
&=\frac{2}{\varepsilon}e^{2C\sqrt{s}-\frac{\varepsilon}{2}s} +\frac{2C}{\varepsilon}\int_s^{+\infty}\sigma^{-\frac{1}{2}}e^{2C\sqrt{\sigma}-\frac{\varepsilon}{2}\sigma} d\sigma\\
&\leq \frac{2}{\varepsilon}e^{2C\sqrt{s}-\frac{\varepsilon}{2}s} +\frac{2C}{\varepsilon}s^{-\frac{1}{2}}\int_s^{+\infty}e^{2C\sqrt{\sigma}-\frac{\varepsilon}{2}\sigma} d\sigma\\
&\leq \frac{2}{\varepsilon}e^{2C\sqrt{s}-\frac{\varepsilon}{2}s} +\frac{2C}{\varepsilon}s^{-\frac{1}{2}}\gamma(s).
\end{split}
\end{equation}
Therefore, there exists $C'>0$ and $s_{**}\geq s_*$ such that, for all $s>s_{**}$, $\gamma(s) \leq C'e^{2C\sqrt{s}-\frac{\varepsilon}{2}s}$. This yields $z(s)=O(\frac{1}{s})$ as $s\rightarrow +\infty$, which closes the proof.
\end{proof}

Now, we can start the proof of Proposition \ref{prp:switch}.
\begin{proof}[Proof of Proposition \ref{prp:switch}]

\begin{enumerate}[label=(\roman*)]
\item We will only prove the estimate for $l_1$ since the proof for $l_0$ is  similar. Using Corollary \ref{cor: dyn I and l tilde} and \eqref{eq:rough}, we see that on $x(s)=l_1(s)e^{-s/2}$ and $y(s)=\overset{\sim}{I}(s)e^{-s/2}$ verify the hypothesis of Lemma \ref{lemma:x,y}. Obtaining two cases:

\textbf{Case 1:} $\bm{l_1(s)\sim}\overset{\sim}{\bm{I}}\bm{(s).}$ Then, \eqref{dyn I & l} gives $l_1\equiv0$, for large time (hence, $l_1(s)=O\left(\frac{\overset{\sim}{I}(s)}{s}\right)$), or $l_1(s)\geq Cs^{-C}e^{\frac{s}{2}}$, which is not possible because \eqref{eq:rough} gives us $l_1(s)\rightarrow 0$ when $s\rightarrow +\infty$.

\textbf{Case 2:} $\bm{l_1(s)=O\left(\frac{\overset{\sim}{I}(s)}{s}\right).}$ Therefore, we are left only with this scenario, concluding the proof of \ref{l0,l1}.
\item Let's make the following assumption:
\begin{equation}\label{control l_n by I}
\forall n\in \N,\ l_n(s)=O\left(\frac{\overset{\sim}{I}(s)}{s}\right).
\end{equation}
We consider two cases in the following:

\textbf{Case 1: \eqref{control l_n by I} is true.} Then, using \eqref{dyn I & l n>1}, we get, for any $n\leq n_0$, that there exists a constant $C_n\in \R$ such that, for $s$ large enough,
$$\overset{\sim}{I'}(s)\leq \left(1-\frac{n+1}{2}+\frac{C_n}{\sqrt{s}}\right)\overset{\sim}{I}(s).$$
By solving the differential inequality, we obtain, for a certain $\sigma_n,C_n>0$,
\begin{equation*}
\forall s\geq \sigma_n,\ \overset{\sim}{I}(s)\leq C \exp\left(\left(\frac{1-n}{2}\right)s+C_n\sqrt{s}\right),
\end{equation*}
which is valid for all $n\leq n_0$. Thus,
\begin{equation}\label{inequality n_0}
\forall s\geq \sigma_{n_0},\ \overset{\sim}{I}(s)\leq C \exp\left(\left(\frac{1-n_0}{2}\right)s+C\sqrt{s}\right).
\end{equation}
For $n>n_0$, we have
$$\overset{\sim}{I'}(s)\leq \left(-\nu+\frac{C_n}{\sqrt{s}}\right)\overset{\sim}{I}(s).$$
We obtain, for a certain $\sigma_1,C>0$,
\begin{equation*}
\forall s\geq \sigma_1=\max(\sigma_{n_0},\sigma_1),\ \overset{\sim}{I}(s)\leq C \exp(-\nu s+C\sqrt{s}).
\end{equation*}

Together with \ref{inequality n_0}, we obtain
\begin{equation*}
\forall s\geq \sigma,\ \overset{\sim}{I}(s)\leq C \exp\left(-\max\left(\nu,\frac{n_0-1}{2}\right)s+C\sqrt{s}\right).
\end{equation*}
Since $\frac{n_0-1}{2}\leq \nu$, we may conclude \eqref{inequality 1}.

\textbf{Case 2: \eqref{control l_n by I} is false.} Then, we take the smallest $n\in \N$ such that \eqref{control l_n by I} does not hold. We know from $(i)$ that $n\geq 2$.

\begin{itemize}
\item \textbf{If} $\mathbf{n\leq n_0}$, we have
\begin{equation}\label{control k<n}
\forall k\in \{0,...,n-1\}, l_k=O\left(\frac{\overset{\sim}{I}(s)}{s}\right).
\end{equation}
We can see that $x(s)=\exp((\frac{n}{2}-1)s)l_n(s)$ and $y(s)=\exp((\frac{n}{2}-1)s)I(s)$ satisfy the assumptions of Lemma \ref{lemma:x,y}. Therefore, we have either $l_n(s)\sim \overset{\sim}{I}(s)$ or there exists $C>0$ so that $l_n(s)\le \frac{C}{s}\overset{\sim}{I}(s)$ which is not possible. Thus, $l_n(s)\sim \overset{\sim}{I}(s)$. Using again \eqref{dyn I & l}, we see that
\begin{equation}\label{edo l_n}
\left|l'_n(s)+\left(\frac{n}{2}-1\right)l_n(s)\right|\leq \frac{C_n}{s}l_n(s),
\end{equation}
which gives us \eqref{inequality 2}.

Let us prove \eqref{inequality m & n}. We already obtained it for $m<n$ from \eqref{control k<n}. For $m>n$, using the fact that $\overset{\sim}{I}(s)\sim l_n(s)$ on \eqref{dyn I & l}
and multiplying by $\exp(-(\frac{m}{2}-1)(s-s_3))$ for $s_3\geq s_2$, we get
\begin{equation}\label{control expression l_m}
l_m(s)\leq e^{-(\frac{m}{2}-1)(s-s_0)}l_m(s_0)+C(m,n)\int_{s_3}^s\frac{l_n(t)}{t}e^{-(\frac{m}{2}-1)(s-t)} dt.
\end{equation}

Since $m>n\geq 2$, then  we have, from \eqref{inequality 2}, that
\begin{equation}\label{inequality exp}
e^{-(\frac{m}{2}-1)(s-s_3)}l_m(s_3)\leq O\left(\frac{l_n(s)}{s}\right).
\end{equation}
For the integral part of the right-hand, by integration by part, we have
\begin{equation}\label{ipp}
\begin{aligned}
\int_{s_3}^s&\frac{l_n(t)}{t}e^{-(\frac{m}{2}-1)(s-t)} d t\\
&\leq   \frac{2}{m-2}\left(\frac{l_n(s)}{s} -e^{-(\frac{m}{2}-1)(s-s_3)}\frac{l_n(s_3)}{s_3}-\int^s_{s_3}e^{-(\frac{m}{2}-1)(s-t)}\left(\frac{l_n'(t)}{t}-\frac{l_n(t)}{t^2}\right)d t\right).
\end{aligned}
\end{equation}

One may see, from \eqref{dyn I & l}, that 
\begin{equation}
\begin{split}
\left|\int^s_{s_3}e^{-(\frac{m}{2}-1)(s-t)}\frac{l'_n(t)}{t}d t+\left(\frac{n}{2}-1\right)\int^s_{s_3}e^{-(\frac{m}{2}-1)(s-t)}\frac{l_n(t)}{t}d t\right|&\\
\leq \int^s_{s_3}&e^{-(\frac{m}{2}-1)(s-t)}\frac{l_n(t)}{t^2}d t\\
\leq \frac{1}{s_3}&\int^s_{s_3}e^{-(\frac{m}{2}-1)(s-t)}\frac{l_n(t)}{t}d t.
\end{split}
\end{equation}

Then, \eqref{ipp} becomes
\begin{equation}
\begin{aligned}
\left|\int_{s_3}^s\frac{l_n(t)}{t}e^{-(\frac{m}{2}-1)(s-t)} d t-\frac{n-2}{m-2}\int_{s_3}^s\frac{l_n(t)}{t}e^{-(\frac{m}{2}-1)(s-t)} d t\right|&\\
\leq   C(m,n)\frac{l_n(s)}{s} +\frac{C(m,n)}{s_3}&\int^s_{s_3}e^{-(\frac{m}{2}-1)(s-t)}\frac{l_n(t)}{t}d t.
\end{aligned}
\end{equation}

Therefore, for $s_3$ large enough,
$$\int_{s_3}^s\frac{l_n(t)}{t}e^{-(\frac{m}{2}-1)(s-t)} d t\leq C(m,n)\frac{l_n(s)}{s}.$$

Together with \eqref{inequality exp}. \eqref{control expression l_m} becomes

$$l_m(s)=O\left(\frac{l_n(s)}{s}\right).$$

\item \textbf{If} $\mathbf{n>n_0}$, then 
\begin{equation}
\forall k\in \{0,...,n_0\}, l_k=O\left(\frac{\overset{\sim}{I}(s)}{s}\right).
\end{equation}
Therefore, from \eqref{dyn I & l n>1}, we obtain
\begin{equation}
\overset{\sim}{I'}(s)\leq -\nu\overset{\sim}{I}(s)+\frac{C}{\sqrt{s}}\overset{\sim}{I}(s).
\end{equation}
Thus, \eqref{inequality 1} follows. Which concludes the proof of Proposition \ref{prp:switch}.
\end{itemize}
\end{enumerate}
\end{proof}

\subsection{Further refinement related to the dominating component}
Following Proposition \ref{prp:switch}, we further investigate the cases where $I \sim l_n$, with  $n=2$ and $n=3$.
\begin{prp}[On the dominating component]\label{prp:l_2,l_3}
$ $
\begin{enumerate}[label=(\roman*)]
\item If  $2\leq n_0$ and $\overset{\sim}{I}(s)\sim l_2(s)$, then $l'_2(s)=-\frac{2}{s}l_2(s)+O\left(\frac{l_2(s)}{s^{3/2}}\right)$ and there exists $C_2>0$ such that $\ l_2(s)=\frac{C_2}{s^2}+O\left(\frac{1}{s^{5/2}}\right)$.\label{item 1}
\item If $3\leq n_0$ and $\overset{\sim}{I}(s)\sim l_3(s)$, then $l_3'(s)=-\left(\frac{1}{2}+\frac{3}{s}\right)l_3(s)+O\left(\frac{l_3(s)}{s^{3/2}}\right)$ and there exists $C_3>0$ such that $\ l_3(s)=\frac{C_3}{s^3}e^{-s/2}+O\left(\frac{e^{-s/2}}{s^{7/2}}\right)$.\label{item 2}
\end{enumerate}
\end{prp}

\begin{proof}
$ $
\begin{enumerate}[label=(\roman*)]
\item We assume that $\overset{\sim}{I}(s)\sim l_2(s)$. Because of \eqref{eq: proof lemma} for $|\beta|=2$, we just need to study
$$M_1(s)=\int \alpha_1(y,s)g(y,s)k_\beta(y)\rho(y)d y,$$
$$M_2(s)=\int \alpha_2(y,s)k_\beta(y)\rho(y)d y.$$
We write
\begin{equation*}
M_1(s)=\int \frac{1}{2s}\left(N-\frac{|y|^2}{2}\right)g(y,s)k_\beta(y)\rho(y)dy+\int \gamma(y,s) g(y,s)k_\beta(y)\rho(y)d y,
\end{equation*}
with $\gamma(y,s)=\alpha_1( y,s)- \frac{1}{2s}(N-\frac{|y|^2}{2}).$ Using \eqref{estimation alpha_1 3}, we have that
\begin{equation*}
\left|\int \gamma(y,s) g(y,s)k_\beta(y)\rho(y)d y\right| \leq \frac{C}{s^{3/2}}I(s)\|(1+|Id|^4)k_\beta\|_{L^2_\rho}.
\end{equation*}
Since $l_2(s)\sim \overset{\sim}{I}(s)=I(s)+se^{-\nu s}\geq I(s)$, then 
\begin{equation}\label{M1 part 2}
\left|\int \gamma(y,s) g(y,s)k_\beta(y)\rho(y)d y\right|m=O\left( \frac{l_2(s)}{s^{3/2}}\right).
\end{equation}
With similar computations  as what has been done in \cite[p.1200]{KZ2000}, we have
$$\int \frac{1}{2s}\left(N-\frac{|y|^2}{2}\right)g(y,s)k_\beta(y)\rho(y)dy= -\frac{2}{s}g_\beta(s).$$
Thus,
$$M_1(s)=-\frac{2}{s}g_\beta(s)+O\left(\frac{l_2(s)}{s^{3/2}}\right).$$
For $M_2$, since $n_0\geq 2$, then from \eqref{estimation alpha_2} and \eqref{inequality 2} for $n=2$, we obtain
\begin{equation}\label{M2}
M_2(s)\leq \int |\alpha_2(y,s)|k_\beta(y)\rho(y)d y\leq Ce^{-\nu s}\leq  Cs^C e^{-\nu s} l_2(s).
\end{equation}
Together with \eqref{eq: proof lemma}, we obtain 
$$l'_2(s)=-\frac{2}{s}l_2(s)+O\left(\frac{l_2(s)}{s^{3/2}}\right).$$
Therefore,
\begin{equation}
\begin{split}
l_2(s)&=l_2(s_0)\left(\frac{s_0}{s}\right)^2exp\left(\int^s_{s_0}\frac{\phi(\sigma)}{\sigma^{3/2}} d\sigma\right)\\
&=l_2(s_0)\left(\frac{s_0}{s}\right)^2exp\left(\left(\int^{+\infty}_{s_0}-\int^{+\infty}_{s}\right)\frac{\phi(\sigma)}{\sigma^{3/2}} d\sigma\right),\\
\end{split}
\end{equation}
where $\phi$ is bounded. Since $\int^{+\infty}_{s}\frac{\phi(\sigma)}{\sigma^{3/2}} d\sigma=O(\frac{1}{\sqrt{s}})$. Therefore, using a Taylor expansion, we obtain
$$l_2(s)=l_2(s_0)\left(\frac{s_0}{s}\right)^2exp\left(\int^{+\infty}_{s_0}\frac{\phi(\sigma)}{\sigma^{3/2}} d\sigma\right)+O\left(\frac{1}{s^{5/2}}\right).$$
From \eqref{inequality 2}, there exists $C_1>0$ such that \ref{item 1} holds.

\item Now, we assume that $\overset{\sim}{I}(s)\sim l_3(s)$. We know again from \cite[p.1200]{KZ2000} that for $|\beta|=3$, we have
$$\int \frac{1}{2s}\left(N-\frac{|y|^2}{2}\right)g(y,s)k_\beta(y)\rho(y)dy= -\frac{3}{s}g_\beta(s).$$
Together with \eqref{M1 part 2}, we get
$$M_1(s)=-\frac{3}{s}g_\beta(s)+o\left(\frac{l_3(s)}{s^{3/2}}\right).$$
For $M_2$, using again Cauchy-Schwarz and \eqref{inequality 2}, we obtain 
\begin{equation}
M_2(s)\leq Ce^{-\nu s}\leq Cs^C e^{-(\nu-\frac{1}{2}) s} l_3(s).
\end{equation}
But since $3\leq n_0$, we have from \ref{def n0} that $\nu \geq 1$. Using \eqref{M1 part 2} and \eqref{eq: proof lemma}, we get
$$g_\beta'(s)=-\left(\frac{1}{2}+\frac{3}{s}\right)g_\beta(s)+O\left(\frac{l_3(s)}{s^{3/2}}\right).$$
Thus, 
$$l_3'(s)=-\left(\frac{1}{2}+\frac{3}{s}\right)l_3(s)+O\left(\frac{l_3(s)}{s^{3/2}}\right).$$
With the same arguments as for $l_2$, we obtain \ref{item 2}, closing the proof.
\end{enumerate}
\end{proof}

\subsection{Conclusion of the proof of Proposition \ref{prp:g in L^2_rho}}
We can now deduce directly Proposition \ref{estimation L^2_rho} from the following Corollary, which is an immediate consequence of Propositions \ref{prp:switch} and \ref{prp:l_2,l_3}.
\begin{cor}\label{corollary: switch dom comp}
Only three cases may occur:
\begin{enumerate}[label=(\roman*)]
\item $n_0\geq 2$ with $\overset{\sim}{I}(s)\sim l_2(s)$ and there exists a constant $C_2>0$ such that $l_2(s)=\frac{C_2}{s^2}+O\left(\frac{1}{s^{5/2}}\right)$.\label{item 1 conclusion}
\item $\overset{\sim}{I}(s)=O\left(\frac{e^{-s/2}}{s^3}\right)$.\label{item 2 conclusion}
\item $\overset{\sim}{I}(s)=O(e^{-\nu s+C\sqrt{s}})$, where $\nu$ is defined in \eqref{estimation dim 1}.\label{item 3 conclusion}
\end{enumerate}
\end{cor}

\begin{proof}
From Proposition \ref{prp:switch}, we have one of the following cases:
\begin{itemize}
\item either 
\begin{equation*}
\forall s\geq \sigma,\ \overset{\sim}{I}(s)\leq C \exp(-\nu s+C\sqrt{s}),
\end{equation*}
which gives us \ref{item 3 conclusion},

\item or there exists $n\leq n_0$ such that $\overset{\sim}{I}(s)\sim l_n(s)$. Since $n_0>1$, then:

\begin{itemize}
\item If $n=2$, from Proposition \ref{prp:l_2,l_3}, we have \ref{item 1 conclusion}. 
\item If  $n=3$, again from Proposition \ref{prp:l_2,l_3}, we have \ref{item 2 conclusion}.
\item If $n\geq 4$, from \eqref{inequality 2}, we have \ref{item 2 conclusion}.
\end{itemize}
\end{itemize}
\end{proof}

Now, we are ready to give the conclusion of Proposition \ref{prp:g in L^2_rho}.
\begin{proof}[Proof of Proposition \ref{prp:g in L^2_rho}.]

To get \eqref{estimation L^2_rho N=1}, we will assume that $N=1$. We can see that if the cases \ref{item 2 conclusion} and \ref{item 3 conclusion} of Corollary \ref{corollary: switch dom comp} are verified, then \eqref{estimation L^2_rho N=1} holds, for $\sigma_0=0$.
Thus, we focus on case 1. We define 
\begin{equation}\label{g bar}
\overline{g}(y,s)=w_1(y,s)-w_2(y,s+\sigma_0),
\end{equation}
for some $\sigma_0$ which is going to be determined afterwards. Since $w_2(.,.+\sigma_0)$ is still a solution for \eqref{eq1:cv}, then the previous results proved for $g$ holds also for $\bar g$. In the following, we will add a $\bar{\,}$ on every item related to $\bar{g}$. 

We have the following Lemma, which proves that the first case of the previous Corollary will never be verified by $\bar g$, provided that $\sigma_0$ is well chosen.

\begin{lemma}\label{lemma contradiction}
There exists $\sigma_0\in \R$ such that $\bar{l}_2(s)=o(\frac{1}{s^2})$.
\end{lemma}
Note that Corollary \ref{corollary: switch dom comp} holds not only for $g$ defined in \eqref{defw1w2g} but also for $\bar{g}$ \eqref{g bar}. From this Lemma together with Corollary \ref{corollary: switch dom comp} for $\bar{g}$, we may deduce that only case 2 or 3 may occur for $\bar{g}$ in dimension 1. This concludes the proof of Proposition \ref{prp:g in L^2_rho}, if we justify Lemma \ref{lemma contradiction}.
\begin{proof}[Proof of Lemma \ref{lemma contradiction}]
The proof is similar to what has been done in \cite[Lemma 2.11]{KZ2000} but it relies on the following Lemma:
 
\begin{lemma}\label{lemma:FL1999}
We expand $w_2$ as in \eqref{decomposition g},
$$w_2(y,s )= w_{2,0}(s)+ w_{2,1}(s)y + w_{2,2}(s)( y^2-2) + w_{2,-} (y,s ).$$
Then, there exists a constant $C>0$ and $\sigma_3>0$ such that
\begin{equation}\label{eq w_2}
\forall s\geq \sigma_3,\ \partial_s w_{2,2}(s)=\frac{1}{C}w_{2,2}(s)^2+O(w_{2,2}(s)^3) \text{ and } w_{2,2}(s) \sim - \frac{C}{s}.
\end{equation}
\end{lemma}
\begin{proof}
The proof is done in \cite[Corollary 4.2 and Proposition 5.1]{FL1993} for the case $h\equiv 0$. The result remains the same in the general case. However, to convince the reader, we give more details of the proof in Appendix \ref{Lemma FK1999}.
\end{proof}
This concludes the proof Lemma \ref{lemma contradiction}.
\end{proof}
This also concludes the proof of Proposition\ref{prp:g in L^2_rho}.
\end{proof}

\section{$L^\infty(\R^N)$ estimate in the blow-up region and near the blow-up point }\label{L infinity estimate in blow-up region}
In this section, we will use the $L^2_\rho$ in Proposition \ref{prp:g in L^2_rho} to obtain an estimate $L^\infty(\R^N)$ norm on the set $\{|y|\leq K_0\sqrt{s}\}$.
\begin{prp}[$L^\infty(\R^N)$ estimate in the blow-up region]\label{prop:estimation g}
For all $K_0>0$, there exist a $s_4\geq s_3$ and a constant $C(K_0)>0$ such that
$$\forall s\geq s_4,\ \forall |y|\leq K_0\sqrt{s},\ |g(y,s)|\leq \frac{C(K_0)}{s}.$$
For $N=1$, there exists $\sigma_0\in \R$ such that 
$$\forall K_0\in \R^{*+},\ \exists C(K_0)>0,\ \forall s\geq s_4,\ \forall |y|\leq K_0\sqrt{s},$$
$$|w_{1}(y,s)-w_{2}(y,s+\sigma_0)|\leq C(K_0)\max \left(s^{-\frac{3}{2}}e^{-s/2},s^{\nu-2}e^{-\nu s+C\sqrt{s}}\right).$$
\end{prp}
\begin{proof}
The proof is very similar to what has been done in \cite{KZ2000}, we will check the differences while recalling the proof.


From Lemma \ref{lemma:estimation alpha_1/2}, we define $Z(y,s)=|g(y,s)|e^{-2C_0\sqrt{s}}+e^{-\nu s}$, where $C_0$ is defined in \eqref{alpha_1 sqrt s}. If $Z(y,s)=g(y,s)e^{-2C_0\sqrt{s}}+e^{-\nu s}$, for some $s$, we can see that 
$$\partial_s Z(y,s)=\mathcal{L}Z(y,s)+g(y,s)e^{-2C_0 \sqrt{s}}\left(\alpha_1(y,s)-\frac{C_0}{\sqrt{s}}\right)+\alpha_2(y,s)e^{-2C_0\sqrt{s}}-\nu e^{-\nu s}.$$
Using \eqref{alpha_1 sqrt s} and \eqref{estimation alpha_2}, we obtain 
$$\partial_s Z(y,s)\leq \mathcal{L}Z(y,s)+e^{-\nu s}(e^{-2C_0\sqrt{s}}-\nu )\leq \mathcal{L}Z(y,s),$$
where the last inequality is valid for all $s\geq s_*$, for some $s_*$ large enough.

If $Z(y,s)=-g(y,s)e^{-2C_0\sqrt{s}}+e^{-\nu s}$, for some $s$, then we can see that 
$$\partial_s Z(y,s)=\mathcal{L}Z(y,s)-g(y,s)e^{-2C_0 \sqrt{s}}\left(\alpha_1(y,s)-\frac{C_0}{\sqrt{s}}\right)-\alpha_2(y,s)e^{-2C_0\sqrt{s}}-\nu e^{-\nu s}.$$
Using again \eqref{alpha_1 sqrt s} and \eqref{estimation alpha_2}, we obtain 
$$\partial_s Z(y,s)\leq \mathcal{L}Z(y,s)+e^{-\nu s}(e^{-2C_0\sqrt{s}}-\nu )\leq \mathcal{L}Z(y,s).$$
Together with Proposition \ref{prp:g in L^2_rho}, we have the two followings:
$$\partial_s Z\leq \mathcal{L}Z,$$
$$\|Z(s)\|_{L^2_\rho}=O\left(\frac{e^{-2C_0\sqrt{s}}}{s^2}\right).$$
We have the following lemma which allows us to conclude.

\begin{lemma}\label{lemma Z}
Let be $Z(s)\in L^2_\rho(\R^N)$, for all $s\geq s_4$, such that 
$$\partial_s Z\leq \mathcal{L}Z.$$
Therefore, there exists $s_4(K_0)\geq s_*$ such that
$$\forall s\geq s_4(K_0), |y|\leq K_0\sqrt{s}, Z(y,s)\leq C(K_0)s \|Z(s-\log s)\|_{L^2_\rho}.$$
\end{lemma}
\begin{proof}
See proof of Proposition 2.13 in \cite{KZ2000} and also \cite{V1992}.
\end{proof}
Using Lemma \ref{lemma Z}, we obtain for all $s\geq s_4$, where $s_4$ large enough and $|y|\leq \frac{K_0\sqrt{s}}{4}$, that
$$0\leq Z(y,s)\leq C_1(K_0)s^{-1}e^{-2C_0\sqrt{s-K_0}}.$$
Thus, we get
$$\underset{|y|\leq \frac{K_0\sqrt{s}}{4}}{\sup}|g(y,s)|\leq C_1(K_0)s^{-1}e^{2C_0\sqrt{s}(1-\sqrt{\frac{K_0}{s}})}=C_1(K_0)e^{2C_0\sqrt{s}\ O(\frac{1}{s})}\leq C_2(K_0) s^{-1}.$$
Same arguments for the case of $N=1$ can be applied, which concludes the proof.
\end{proof}

One may see that we can derive the following consequence of Proposition.\ref{prop:estimation g}.
\begin{cor}\label{cor: conclusion}
For all $K_0>0$, there exist $\delta_0\in (0,T_1)$ and $C(K_0)>0$ such that for all $t\in (T_1-\delta_0,T_1)$ and $x\in B(a_1,K_0\sqrt{(T_1-t)|\log(T_1-t)|})$ we have

\begin{equation}\label{inequality: error2}
|u_1(x,t)- \overset{\sim}{u}_2(x,t)| \leq C(K_0)(T_1-t)^{-\frac{1}{p-1}}|\log(T_1-t)|^{-1},
\end{equation}
where $\overset{\sim}{u}_2(x,t)=\mathcal{T}_{a_1-a_2,T_1-T_2}u_2(x,t)$. Moreover, if $N=1$, then there exists $\lambda>0$ such that
\begin{equation}\label{inequality: error3}
|u_1(x,t)-\overline{u}_2(x,t)|\leq C(K_0)\max\left(\frac{(T_1-t)^{\frac{1}{2}-\frac{1}{p-1}}}{|\log(T_1-t)|^{\frac{3}{2}}}, \frac{(T_1-t)^{\nu-\frac{1}{p-1}}}{|\log(T-t)|^{2-\nu}}\exp\left(C\sqrt{-\log(T-t)}\right)\right),
\end{equation}
where $ \overline{u}_2(x,t)=\mathcal{T}_{a_1,T_1}\mathcal{D}_{\lambda}\mathcal{T}_{-a_2,-T_2}u_2(x,t)$.
\end{cor}

\begin{proof}
One may see that estimate \eqref{inequality: error2} is a direct consequence of Proposition \ref{prop:estimation g}. Now, for \eqref{inequality: error3}, we consider $\lambda\in \R^+$ such that , for all $(x,t)\in \R^N\times[0,T_1)$,
$$\overline{u}_2(x,t):=\lambda^{\frac{2}{p-1}}u_2(\lambda(x-a_1)+a_2,T_1-\lambda^2(T_1-t)).$$
We denote by $\overline{w}$  its similarity variables version (see \eqref{cv}), i.e.
\begin{equation}
\begin{split}
\overline{w}(y,s)&:=e^{-\frac{s}{p-1}}\overline{u}_2(a_1+ye^{-\frac{s}{2}},T_1-e^{-s})\\
&=e^{-(s+2\ln\lambda)\frac{p-1}{2}}u_2(a_1+ye^{-\frac{s+2\ln\lambda}{2}},T_1-e^{-(s+2\ln\lambda )})\\
&=w(y,s+\sigma_0),
\end{split}
\end{equation}
where $\sigma_0=2\ln\lambda$. Thus, with Proposition \ref{prop:estimation g}, we obtain \eqref{inequality: error3}, concluding the proof of Corollary \ref{cor: conclusion}.
\end{proof}

\section{No blow-up under some threshold}\label{No blow up threshold}

In this section, we prove the non blow-up property under some threshold (see Propostion \ref{prp bound v}). Without loss of generality, we may assume again $a_1=0$, since the equation is invariant by translations. We fix $K_0>0$ and $\beta\in (0,K_0/2)$. We take $\delta_0$ and $C(K_0)$ associated with $K_0$ by Corollary \ref{cor: conclusion}. From that statement, we are left with the case where $t\in (T_1-\delta_0,T_1)$ and

\begin{equation}\label{region outside}
|x|\geq K_0\sqrt{(T_1-t)|\log(T_1-t)|)}.
\end{equation}
Furthermore, we define for all $|x|\leq \frac{K_0}{2}\sqrt{T_1|\log T_1|}$, $t(x)\in [0,T_1)$ such that
\begin{equation}\label{def t(x)}
|x|= \frac{K_0}{2}\sqrt{(T_1-t(x))|\log (T_1-t(x))|}.
\end{equation}

\begin{figure}[h!]
    \centering
\begin{tikzpicture}[scale=1.5]
    \fill[red!30, scale=0.5, domain=0:4.7675, variable=\x] plot({\x},{-(\x/3)*2)*(\x/3)+5});
    \fill[red!30, scale=0.5] (0,0)--(4.7675,0)--(0,5) -- cycle;
    \draw[thick,->] (-0.5,0) -- (3.5,0);
    \draw [right] (3.5,0) node{$|x|$};
    \draw[->] (0,-0.5) -- (0,3.5);
    \draw [above](0,3.5) node{$t$};
    \draw [blue](-0.5,2.5) -- (3.5,2.5);
    \draw [left] (0,2.65) node{$T$};
    \fill [yellow!20](0,2) rectangle (0.95,2.5);
    \draw(0,2) rectangle (0.95,2.5);
    \draw [left] (0,2) node{$T-\delta$};
    \draw [decorate,decoration = {brace, mirror}] (0,1.95) --  (0.95,1.95);
    \draw (0.6,1.7) node{$|x|<\varepsilon$};
    \draw[<-] (0.8,2.3) -- (2,3);
    \draw [right] (2,3) node{We want to estimate this region};
    \draw[scale=0.5, domain=0:4.768, variable=\x] plot({\x},{-(\x/3)*2)*(\x/3)+5});
    \draw[<-] (1.7,1) -- (2.5,1);
    \draw [right] (2.5,1) node{$|x|\leq K_0\sqrt{(T-t)|\log(T-t)|}$};
    \draw[scale=0.5, domain=0:3.95, variable=\x] plot({\x},{-(\x/2.5)*2)*(\x/2.5)+5});
    \draw[<-] (1.78,0.5) -- (2.5,0.5);
    \draw [right] (2.5,0.5) node{$|x|=\frac{K_0}{2}\sqrt{(T-t(x))|\log(T-t(x))|}$};
\end{tikzpicture}
\end{figure}

Consider $u_i\in S_{0,T_i,h_i}$, for $i=1,2$. Introducing $\tau=\frac{t-t(x)}{T_1-t(x)}$, we see that $\tau\in [0,1)$.
Then, for some $\beta>0$, we define for all $\tau\in [0,1)$ and $|\xi|\leq \beta (\log(T_1-t(x)))^{\frac{1}{4}}$, the following:
\begin{equation}\label{def: v}
\begin{split}
&v_1(\xi,\tau)=(T_1-t(x))^{\frac{1}{p-1}}u_1(x+\xi\sqrt{T_1-t(x)},t(x)+\tau (T_1-t(x))),\\
&v_2(\xi,\tau)=(T_1-t(x))^{\frac{1}{p-1}}\overset{\sim}{u}_2(x+\xi\sqrt{T_1-t(x)},t(x)+\tau (T_1-t(x))),
\end{split}
\end{equation}
where $\overset{\sim}{u}_2$ is given in \eqref{inequality: error2}. Note that $v_1$ and $v_2$ depend on the variables $\xi$ and $\tau$, and that $t\in [t(x),T_1]$ and $x$ verifying \eqref{region outside} are just parameters. We get for all $i\in\{1,2\}$ and $\tau\in[0,1),|\xi|\leq\beta (|\log(T_1-t(x))|)^{\frac{1}{4}}$, the following equation satisfied by $v_i$:
\begin{equation}\label{eq vi}
\partial_\tau v_{i}=\Delta v_{i}+|v_{i}|^{p-1}v_{i}+(T_1-t(x))^{\frac{p}{p-1}}h_i\left((T_1-t(x))^{-\frac{1}{p-1}}v_i,(T-t(x))^{-\frac{p+1}{2(p-1)}}\nabla v_i\right).
\end{equation}

\subsection{No blow-up under a threshold}
From now on, until the end of Section \ref{No blow up threshold},$u$ stands for $u_1$ or $\overset{\sim}{u}_2$ and $v$ for $v_1$ or $v_2$. We have the following lemma, which gives us a rough estimate on $v$ and $\nabla v$.
\begin{lemma}[Rough bound on $v$ and its gradient]\label{Lemma rough estimate}
For any $\varepsilon_0 > 0$, there exist $K_0 > 0$ and $r_0 > 0$ such that if $x$ satisfies \eqref{region outside} and $|x| < r_0$, then we have
$$\forall |\xi|\leq\beta (|\log(T_1-t(x))|)^{\frac{1}{4}},\tau\in[0,1), |v(\xi,\tau)|\leq \varepsilon_0(1-\tau)^{-\frac{1}{p-1}}, $$
and
$$\forall \xi\in \R^N, \tau \in [0,1),|\nabla v(\xi,\tau)|\leq \varepsilon_0 (1-\tau)^{-\frac{1}{p-1}-\frac{1}{2}}.$$

\end{lemma}

\begin{proof}
From \eqref{behaviour1}, we have
\begin{equation}\label{profile u}
\underset{x'\in \R^N}{\sup}\left|(T_1-t)^{\frac{1}{p-1}}u(x',t)-f\left(\frac{x'}{\sqrt{(T_1-t)|\log(T_1-t)|}}\right)\right|\leq \frac{C}{\sqrt{|\log(T_1-t)|}}.
\end{equation}
Since we have $t=t(x)+\tau (T_1-t(x))$, then $\frac{T_1-t}{T_1-t(x)}=1-\tau$. Therefore, \eqref{profile u} becomes
\begin{equation}\label{profile v}
\begin{split}
\underset{\xi\in \R^N}{\sup}\left|(1-\tau)^{\frac{1}{p-1}}v(\xi,\tau)-f\left(z\right)\right|&\leq \frac{C}{\sqrt{|\log(T_1-t)|}}\\
&\leq \frac{C}{\sqrt{|\log(T_1-t(x))|}},
\end{split}
\end{equation}
where $z=\frac{x+\xi\sqrt{T_1-t(x)}}{\sqrt{(T_1-t)|\log(T-t)|}}$ with $|\xi|\leq\beta (|\log(T_1-t(x))|)^{\frac{1}{4}}$. Using \eqref{def t(x)}, we get
\begin{equation}\label{inequality z}
\begin{split}
|z|&\geq\frac{K_0}{2}\frac{\sqrt{(T_1-t(x))|\log(T_1-t(x))|}}{\sqrt{(T_1-t)|\log(T_1-t)|}}-\frac{|\xi|\sqrt{T_1-t(x)}}{\sqrt{(T_1-t)|\log(T_1-t)|}}\\
&\geq\frac{\sqrt{(T_1-t(x))|\log(T_1-t(x))|}}{\sqrt{(T_1-t)|\log(T_1-t)|}}\left(\frac{K_0}{2}-\frac{|\xi|}{\sqrt{|\log(T_1-t(x))|}}\right)\\
&\geq \frac{\sqrt{(T_1-t(x))|\log(T_1-t(x))|}}{\sqrt{(T_1-t)|\log(T_1-t)|}}\left(\frac{K_0}{2}-\frac{\beta}{|\log(T_1-t(x))|^{\frac{1}{4}}}\right).
\end{split}
\end{equation}
Since $t\geq t(x)$, then
\begin{equation}\label{z > K_0 beta}
|z|\geq \left(\frac{K_0}{2}-\frac{\beta}{|\log(T_1-t(x))|^{\frac{1}{4}}}\right).
\end{equation}
For small enough $|x|$, we have 
$$|z|\geq \frac{K_0}{4}.$$
Therefore, we have $f(z)\leq f(\frac{K_0}{4})$. Taking $K_0$ sufficiently big and together with \eqref{profile v}, we have for all $\tau \in [0,1)$ and $|\xi|\leq\beta (|\log(T_1-t(x))|)^{\frac{1}{4}}$, the following:
\begin{equation}\label{rough bound v}
(1-\tau)^{\frac{1}{p-1}}|v(\xi,\tau)|\leq \varepsilon_0,
\end{equation}
for $\varepsilon_0>0$ as small as we want. Again from \eqref{behaviour1}, we have
\begin{equation*}
\begin{split}
\underset{x'\in \R^N}{\sup}\left|(T_1-t)^{\frac{1}{p-1}+\frac{1}{2}}\nabla u(x',t)-\frac{1}{\sqrt{|\log(T_1-t(x))|}}\right.&\left.\nabla f\left(\frac{x'}{\sqrt{(T_1-t)|\log(T_1-t)|}}\right)\right|\\
&\leq \frac{C}{\sqrt{|\log(T_1-t)|}}.
\end{split}
\end{equation*}
Therefore,
\begin{equation*}
\underset{\xi\in \R^N}{\sup}\left|(1-\tau)^{\frac{1}{p-1}+\frac{1}{2}}\nabla v(\xi,\tau)-\frac{1}{\sqrt{|\log(T_1-t(x))|}}\nabla f\left(z\right)\right|\leq \frac{C}{\sqrt{|\log(T_1-t(x))|}}.
\end{equation*}
Thus,
\begin{equation*}
\underset{\xi\in \R^N}{\sup}\left|(1-\tau)^{\frac{1}{p-1}+\frac{1}{2}}\nabla v(\xi,\tau)+\frac{(p-1)f\left(z\right)^p}{2p\sqrt{|\log(T_1-t(x))|}} z\right|\leq \frac{C}{\sqrt{|\log(T_1-t(x))|}}.
\end{equation*}
By taking $|x|$ sufficiently small, we obtain, for all $\xi\in \R^N$ and $\tau \in [0,1)$,
\begin{equation}\label{rough bound grad v}
|\nabla v(\xi,\tau)|\leq \varepsilon_0 (1-\tau)^{-\frac{1}{p-1}-\frac{1}{2}},
\end{equation}
which concludes the proof.
\end{proof}

In the following proposition, we prove a result which is similar in spirit to what has been done in \cite{GK1989} for $h\equiv0$ and in \cite{TZ2015} for the critical case $p=\frac{2p}{p+1}$. The statement of the proposition is as the following:

\begin{prp}[No blow up under a threshold]\label{prp bound v}
We assume that $v$ and $\nabla v$ satisfy the following equations:
\begin{equation}\label{ineq v}
\begin{split}
\partial_\tau v -\Delta v=f(v,\nabla v),
\end{split}
\end{equation}
\begin{equation}\label{ineq grad v}
\begin{split}
\partial_\tau \nabla v -\Delta \nabla v=\bar f(v,\nabla v),
\end{split}
\end{equation}
with $$|f(v,\nabla v)|\leq C(1+|v|^p+|\nabla v|^{\bar{\gamma}}),$$
$$|\bar f(v,\nabla v)-\nabla \bar g(v)|\leq C(1+|\nabla v|+|\nabla v||v|^{p-1}) \text{ and }|\bar g(v)|\leq C|\nabla v|^{\bar{\gamma}},$$
together with the following bounds, for all $ |\xi|<1$ and $\tau\in [0,1)$,
\begin{equation*}\label{bound v & nabla v}
|v(\xi,\tau)|+\sqrt{1-\tau}|\nabla v(\xi,\tau)|\leq \varepsilon_0(1-\tau)^{-\frac{1}{p-1}},
\end{equation*}
for some $\varepsilon_0>0$. Then, there exists $\varepsilon>0$, such that for all $|\xi|\leq \varepsilon,\tau \in [0,1)$,
$$|v(\xi,\tau)|+|\nabla v(\xi,\tau)|\leq C\varepsilon_0,$$
for some constant $C=C(p)>0$.
\end{prp}
One may see that the hypothesis of Proposition \ref{prp bound v} are verified in our case for small enough $|x|$. Therefore, we have the following direct consequence:
\begin{cor}\label{cor bound v, nabla v}
For any $\varepsilon_0 > 0$, there exist $K_0 > 0$ and $r_0 > 0$ such that, for all $x$ satisfying \eqref{region outside} and $|x| < r_0$, for all $|\xi|\leq\beta (|\log(T_1-t(x))|)^{\frac{1}{4}}-1,\tau \in [0,1)$, we have
\begin{equation}\label{bound v}
|v(\xi,\tau)|+|\nabla v(\xi,\tau)|\leq C\varepsilon_0.
\end{equation}
\end{cor}

\subsection{Proof of Proposition \ref{prp bound v}}
For the proof of Proposition \ref{prp bound v}, we will first need few intermediate outcomes.

\begin{lemma}[Duhamel formulation of $v$]\label{duhamel bound v & nabla v}
Under the assumptions \eqref{ineq v} and \eqref{ineq grad v}, we consider the following cut-off function:
$$\chi_r=\chi_0\left(\frac{\xi}{r}\right),$$
where $\chi_0\in C^\infty_c$, $\chi_0\equiv 1$ on $B(0,1/2)$ and $\chi_0 \equiv 0$ on $\R^N \backslash B(0,1)$ and $0<r\leq 1$ . We also consider $w=\nabla v,w_r= \chi_r w$ and $v_r=\chi_r v$. Then, for all $\xi \in \R, \tau \in [0,1)$,
\begin{equation}\label{duhamel bound v}
\begin{split}
\| v_r(\tau)\|_{L^\infty(\R^N)}\leq\ & C\| v_r(0)\|_{L^\infty(\R^N)}+C\int^\tau_0 (\tau-s)^{-\frac{1}{2}}\|v(s)\|_{L^\infty(B_r)}ds\\
&+ C\int^\tau_0 \|v_r(s)\|_{L^\infty(\R^N)}\|v(s)\|_{L^\infty(B_r)}^{p-1}ds+C\int^\tau_0\|w(s)\|^{\bar \gamma}_{L^\infty(B_r)}ds,
\end{split}
\end{equation}
and 

\begin{equation}\label{duhamel bound nabla v}
\begin{split}
\| w_r(\tau)\|_{L^\infty(\R^N)}\leq\ & C\| w_r(0)\|_{L^\infty(\R^N)}+C\int^\tau_0 (\tau-s)^{-\frac{1}{2}}\|w(s)\|_{L^\infty(B_r)}ds+C\int^\tau_0\|w_r(s)\|_{L^\infty(\R^N)}ds\\
&+C\int^\tau_0\|w_r(s)\|_{L^\infty(\R^N)}\|v(s)\|^{p-1}_{L^\infty(B_r)}ds+C\int^\tau_0(\tau-s)^{-\frac{1}{2}}\|w(s)\|^{\bar{\gamma}}_{L^\infty(B_r)}ds.
\end{split}
\end{equation}

\end{lemma}

\begin{proof}
We can see from \eqref{ineq v} that, for all $\xi \in \R, \tau \in [0,1)$,
\begin{equation*}
\begin{split}
\partial_\tau v_r -\Delta v_r-v\Delta \chi_r+2\nabla(v\nabla \chi_r)=\chi_r f(v,\nabla v).
\end{split}
\end{equation*}
Therefore, using the Duhamel formulation, we see that 
\begin{equation*}
\begin{split}
v_r(\tau)= S(\tau) v_r(0)+&\int^\tau_0 S(\tau-s)\left[v(s)\Delta \chi_r-2\nabla(v(s)\nabla \chi_r)\right]ds+\int^\tau_0 S(\tau-s)\chi_r f(v,\nabla v)(s) ds,
\end{split}
\end{equation*}
where $S$ is the heat semigroup kernel. We recall these well-known smoothing effect of the heat semigroup. For all $t>0$ and  $f\in W^{1,\infty}(\R^N)$, we have

\begin{equation}\label{heat kernel prp}
\|S(t)f\|_{L^\infty(\R^N)}\leq \|f\|_{L^\infty(\R^N)} \text{ and } \|\nabla S(t)f\|_{L^\infty(\R^N)}\leq C t^{-1/2}\|f\|_{L^\infty(\R^N)}.
\end{equation}
We obtain
\begin{equation*}
\begin{split}
\| v_r(\tau)\|_{L^\infty(\R^N)}\leq& \| v_r(0)\|_{L^\infty(\R^N)}+C\int^\tau_0 (\tau-s)^{-\frac{1}{2}}\|v(s)\|_{L^\infty(B_r)}ds\\
&+C\int^\tau_0 \|\chi_r f(v,\nabla v)(s)\|_{L^\infty(B_r)}ds.
\end{split}
\end{equation*}
Thus, 
\begin{equation}\label{before duhamel nabla v}
\begin{split}
\| v_r(\tau)\|_{L^\infty(\R^N)}\leq& \| v_r(0)\|_{L^\infty(\R^N)}+C\int^\tau_0 (\tau-s)^{-\frac{1}{2}}\|v(s)\|_{L^\infty(B_r)}ds+C\int^\tau_0 ds\\
&+ C\int^\tau_0 \|v_r(s)\|_{L^\infty(\R^N)}\|v(s)\|_{L^\infty(B_r)}^{p-1}ds+C\int^\tau_0\|w(s)\|^{\bar \gamma}_{L^\infty(B_r)}ds,
\end{split}
\end{equation}
which gives us \eqref{duhamel bound v}. From \eqref{ineq grad v},
\begin{equation*}
\begin{split}
\partial_\tau w_r -\Delta w_r=w\Delta \chi_r-2\nabla(w\nabla \chi_r)+\chi_r\bar f(v,\nabla v).
\end{split}
\end{equation*}
Again with Duhamel formulation, we obtain
\begin{equation*}
\begin{split}
w_r(\tau)= S(\tau) w_r(0)+&\int^\tau_0 S(\tau-s)\left[w\Delta \chi_r-2\nabla(w\nabla \chi_r)\right]ds+\int^\tau_0 S(\tau-s)\chi_r \bar f(v,\nabla v) ds.
\end{split}
\end{equation*}
For the last term of the right side, one may see that
\begin{equation*}
\begin{split}
\left|\int^\tau_0 S(\tau-s)\chi_r \bar f(v,\nabla v) (s)ds\right|&\leq \int^\tau_0 S(\tau-s)\chi_r |\bar f(v,\nabla v)(s)-\nabla  \bar g(v)(s) |ds\\
&\hspace{0.5cm}+\left|\int^\tau_0 S(\tau-s)\chi_r \nabla \bar g(v)(s)ds\right| \\
&\leq C\left(1+\int^\tau_0 S(\tau-s)|w_r(s)|ds+\int^\tau_0 S(\tau-s)|w_r(s)||v(s)|^{p-1}ds\right.\\
&\hspace{0.5cm} \left.+\int^\tau_0 \nabla S(\tau-s) \chi_r |w(s)|^{\bar{\gamma}}ds+\frac{1}{r}\int^\tau_0 S(\tau-s) |\nabla \chi_r|\ |w(s)|^{\bar{\gamma}}ds\right).
\end{split}
\end{equation*}
Using \eqref{heat kernel prp} , we obtain
\begin{equation*}
\begin{split}
\| w_r(\tau)\|_{L^\infty(\R^N)}\leq& \| w_r(0)\|_{L^\infty(\R^N)}+C\int^\tau_0 (\tau-s)^{-\frac{1}{2}}\|w(s)\|_{L^\infty(B_r)}ds+C\int^\tau_0ds\\
&+C\int^\tau_0\|w_r(s)\|_{L^\infty(\R^N)}ds+C\int^\tau_0\|w_r(s)\|_{L^\infty(\R^N)}\|v(s)\|^{p-1}_{L^\infty(B_r)}ds\\
&+C\int^\tau_0(\tau-s)^{-\frac{1}{2}}\|w(s)\|^{\bar{\gamma}}_{L^\infty(B_r)}ds+C\int^\tau_0\|w(s)\|^{\bar{\gamma}}_{L^\infty(B_r)}ds,
\end{split}
\end{equation*}
which gives \eqref{duhamel bound nabla v}, which concludes the proof.
\end{proof}

Let us now recall the following Gronwall Lemma from Giga and Kohn \cite{GK1989}.

\begin{lemma}[Giga and Kohn, \cite{GK1989}]\label{Gronwall GK} If $y, r$ and $q$ are continuous functions defined on $[\tau_0, \tau_1]$ such that
$$y(\tau) \leq y_0 + \int^\tau_{\tau_0} y(s)r(s)ds + \int^\tau_{\tau_0} q(s)ds,\ \tau_0 \leq \tau \leq \tau_1,$$
then
$$y(\tau) \leq \exp \left\{\int^\tau_{\tau_0} r(s )ds \right\} \left[ y_0(\tau) + \int^\tau_{\tau_0} q(s ) \exp \left\{ - \int^s_{\tau_0} r(\sigma)d\sigma \right\} ds \right] .$$
\end{lemma}

\begin{proof}
See in \cite[Lemma 2.3]{GK1989}.
\end{proof}

The following lemma is needed for some computations.
\begin{lemma}[Giga and Kohn, \cite{GK1989}]\label{lemma I bound}
For $0 < \alpha < 1$, $\theta > 0$, and $0 < h < 1$, the integral
$$
I(h) = \int_h^1 (s - h)^{-\alpha} s^{-\theta} \, ds
$$
satisfies
\begin{enumerate}[label=(\roman*)]
    \item 
    $$
    I(h) \leq \left( \frac{1}{1-\alpha} + \frac{1}{\alpha + \theta - 1} \right) h^{1-\alpha-\theta}, \quad \text{if } \alpha + \theta > 1,
    $$
    \item
    $$
    I(h) \leq \frac{1}{1-\alpha} + |\log h|, \quad \text{if } \alpha + \theta = 1,
    $$
    \item
    $$
    I(h) \leq \frac{1}{1-\alpha - \theta}, \quad \text{if } \alpha + \theta < 1.
    $$
\end{enumerate}

\end{lemma}

\begin{proof}
    See \cite[Lemma 2.2]{GK1989}.
\end{proof}

\begin{lemma}\label{lemma iterations}
We consider $0\leq q\leq \frac{1}{p-1}, 0\leq q'\leq \frac{1}{p-1}+\frac{1}{2}$. We assume, for all $\tau\in [0,1)$, that
$$\|v_r(\tau)\|_{L^\infty(\R^N)}\leq C\varepsilon_0(1-\tau)^{-q},$$
$$\|w_r(\tau)\|_{L^\infty(\R^N)}\leq C\varepsilon_0(1-\tau)^{-q'},$$
where $v_r$ and $w_r$ are defined in Lemma \ref{duhamel bound v & nabla v} for some $r\in (0,1]$. Then, the following holds:

\begin{enumerate}[label=(\roman*)]

\item If $q(p-1)<1$, then for any $\mu>0$, there exists $C(\mu)>\frac{e^{-1}}{\mu}$, such that

\begin{equation}\label{q<1 v}
\begin{split}
\| v_{\frac{r}{2}}(\tau)\|_{L^\infty(\R^N)}\leq C(\mu)\varepsilon_0 (1-\tau)^{\min(\frac{1}{2}-q-\mu,0)}+C(\mu)\varepsilon_0^{\bar \gamma}(1-\tau)^{\min(1-q'\bar \gamma-\mu,0)},
\end{split}
\end{equation}

and

\begin{equation}\label{q<1 w}
\begin{split}
\|w_{\frac{r}{2}}(\tau)\|_{L^\infty(\R^N)}\leq& C(\mu)\varepsilon_0(1-\tau)^{\min(\frac{1}{2}-q'-\mu,0)}+C(\mu)\varepsilon_0^{\bar \gamma}(1-\tau)^{\min(\frac{1}{2}-\bar \gamma q'-\mu,0)},
\end{split}
\end{equation}
\item If $q(p-1)=1$, then

\begin{equation}\label{q=1 v}
\begin{split}
\| v_{\frac{r}{2}}(\tau)\|_{L^\infty(\R^N)}\leq C\varepsilon_0(1-\tau)^{-C\varepsilon_0^{p-1}}+C\varepsilon_0 (1-\tau)^{\min(\frac{1}{2}-q,0)}+C\varepsilon_0^{\bar \gamma}(1-\tau)^{\min(1-q'\bar \gamma,0)},
\end{split}
\end{equation}
and
\begin{equation}\label{q=1 w}
\begin{split}
\|w_{\frac{r}{2}}(\tau)\|_{L^\infty(\R^N)}\leq& C\varepsilon_0(1-\tau)^{-C\varepsilon_0^{p-1}}+C\varepsilon_0(1-\tau)^{\min(\frac{1}{2}-q',0)}+C\varepsilon_0^{\bar \gamma}(1-\tau)^{\min(\frac{1}{2}-\bar \gamma q',0)}.
\end{split}
\end{equation}
\end{enumerate}

\end{lemma}

\begin{proof}
We will only do the proof of \eqref{q<1 v} and \eqref{q=1 v} since the same arguments and computations can be applied to obtain \eqref{q<1 w} and \eqref{q=1 w}. Using \eqref{duhamel bound v} with $\frac{r}{2}$ instead of $r$, we have
\begin{equation*}
\begin{split}
\| v_{\frac{r}{2}}(\tau)\|_{L^\infty(\R^N)}\leq& C\| v_{\frac{r}{2}}(0)\|_{L^\infty(\R^N)}+C\int^\tau_0 (\tau-s)^{-\frac{1}{2}}\|v(s)\|_{L^\infty(B_{\frac{r}{2}})}ds+C\int^\tau_0\|w(s)\|^{\bar \gamma}_{L^\infty(B_{\frac{r}{2}})}ds\\
&+ C\int^\tau_0 \|v_{\frac{r}{2}}(s)\|_{L^\infty(\R^N)}\|v(s)\|_{L^\infty(B_{\frac{r}{2}})}^{p-1}ds.\\
\end{split}
\end{equation*}
Therefore, we have
\begin{equation*}
\begin{split}
\| v_{\frac{r}{2}}(\tau)\|_{L^\infty(\R^N)} \leq& C\varepsilon_0+C\varepsilon_0\int^\tau_0 (\tau-s)^{-\frac{1}{2}}(1-s)^{-q}ds+C\varepsilon_0^{\bar \gamma}\int^\tau_0(1-s)^{-q'\bar \gamma}ds\\
&+ C \varepsilon_0^{p-1}\int^\tau_0 \|v_{\frac{r}{2}}(s)\|_{L^\infty(\R^N)}(1-s)^{-q(p-1)}ds.
\end{split}
\end{equation*}
By Lemma \ref{Gronwall GK}, we get
\begin{equation}\label{condition q}
\begin{split}
\| v_{\frac{r}{2}}(\tau)\|_{L^\infty(\R^N)}&\leq \exp\left(C \varepsilon_0^{p-1}\int^\tau_0 (1-s)^{-q(p-1)} ds\right)\bigg[C\varepsilon_0\\
&\ +C\varepsilon_0\int^\tau_0 (\tau-s)^{-\frac{1}{2}}(1-s)^{-q}\exp\left(-C \varepsilon_0^{p-1}\int^\tau_0 (1-\sigma)^{-q(p-1)} d\sigma\right)ds\\
&\ +C\varepsilon_0^{\bar \gamma}\int^\tau_0(1-s)^{-q'\bar \gamma}\exp\left(-C \varepsilon_0^{p-1}\int^\tau_0 (1-\sigma)^{-q(p-1)} d\sigma\right)ds\bigg].
\end{split}
\end{equation}
We assume $q(p-1)<1$, then
\begin{equation*}
\begin{split}
\| v_{\frac{r}{2}}(\tau)\|_{L^\infty(\R^N)}\leq C\varepsilon_0+C\varepsilon_0\int^\tau_0 (\tau-s)^{-\frac{1}{2}}(1-s)^{-q}ds+C\varepsilon_0^{\bar \gamma}\int^\tau_0(1-s)^{-q'\bar \gamma}ds.
\end{split}
\end{equation*}
One may see that

\begin{equation}
 \int^\tau_0(1-s)^{-q'\bar \gamma}ds\leq \left\{
 \begin{split}
     &\frac{1}{q'\bar \gamma-1}(1-\tau)^{1-q'\bar \gamma} &\text{ if } 1-q'\bar \gamma<0,\\
     & \frac{1}{1-q'\bar \gamma}&\text{ if } 1-q'\bar \gamma>0,\\
     &C(\mu)(1-\tau)^{-\mu}&\text{ if } 1-q'\bar \gamma=0,
 \end{split}
\right\}\leq C(\mu)(1-\tau)^{\min(1-q'\bar \gamma-\mu,0)},
\end{equation}
with $\mu>0$ as small as needed and $C(\mu)>\frac{e^{-1}}{\mu}$. From Lemma \ref{lemma I bound}, up to replacing $\sigma=1-s$ and $h=1-\tau$, one may see also that
\begin{equation}
 \int^\tau_0 (\tau-s)^{-\frac{1}{2}}(1-s)^{-q}ds\leq \left\{
 \begin{split}
     &\left(2+\frac{1}{q-\frac{1}{2}}\right)(1-\tau)^{\frac{1}{2}-q} &\text{ if } \frac{1}{2}-q<0,\\
     & \frac{1}{\frac{1}{2}-q}&\text{ if } \frac{1}{2}-q>0,\\
     &C(\mu)(1-\tau)^{-\mu}&\text{ if } \frac{1}{2}-q=0,
 \end{split}
\right\}\leq C(\mu)(1-\tau)^{\min(\frac{1}{2}-q-\mu,0)}.
\end{equation}
Thus,
\begin{equation*}
\begin{split}
\| v_{\frac{r}{2}}(\tau)\|_{L^\infty(\R^N)}\leq C(\mu)\varepsilon_0 (1-\tau)^{\min(\frac{1}{2}-q-\mu,0)}+C(\mu)\varepsilon_0^{\bar \gamma}(1-\tau)^{\min(1-q'\bar \gamma-\mu,0)},
\end{split}
\end{equation*}
which gives \eqref{q<1 v}. Now, we use again \eqref{condition q} but for $q(p-1)=1$. Then,
\begin{equation*}
\begin{split}
\| v_{\frac{r}{2}}(\tau)\|_{L^\infty(\R^N)}\leq (1-\tau)^{-C\varepsilon_0^{p-1}}&\left[C\varepsilon_0+C\varepsilon_0\int^\tau_0 (\tau-s)^{-\frac{1}{2}}(1-s)^{-q+C\varepsilon_0^{p-1}}ds\right.\\
&\ \left.+C\varepsilon_0^{\bar \gamma}\int^\tau_0(1-s)^{-q'\bar \gamma+C\varepsilon_0^{p-1}}ds\right].
\end{split}
\end{equation*}
Thus, with Lemma \ref{lemma I bound} and taking $\varepsilon_0$ sufficiently small, it holds that
\begin{equation*}
\begin{split}
\| v_{\frac{r}{2}}(\tau)\|_{L^\infty(\R^N)}\leq C\varepsilon_0(1-\tau)^{-C\varepsilon_0^{p-1}}+C\varepsilon_0 (1-\tau)^{\min(\frac{1}{2}-q,0)}+C\varepsilon_0^{\bar \gamma}(1-\tau)^{\min(1-q'\bar \gamma,0)},
\end{split}
\end{equation*}
for small enough $\varepsilon_0>0$, obtaining \eqref{q=1 v}. Closing the proof of Lemma \ref{lemma iterations}.
\end{proof}

Now, we are ready to prove Proposition \ref{prp bound v}.
\begin{proof}[Proof of Proposition \ref{prp bound v}]
The proof is based on iterations for well chosen $q$ and $q'$. 

\textbf{Step 1: If} $\mathbf{\bar\gamma\leq 1}$. 
From Lemma \ref{Lemma rough estimate}, we first consider $r_0=1$, $q_0=\frac{1}{p-1}$ and $q'_0=\frac{1}{p-1}+\frac{1}{2}$. One may see, from \eqref{q=1 w}, that

\begin{equation*}
\begin{split}
\|w_{r_1}(\tau)\|_{L^\infty(\R^N)}\leq& C\varepsilon_0(1-\tau)^{-C_0\varepsilon_0^{p-1}}+C\varepsilon_0(1-\tau)^{\min(\frac{1}{2}-q_0',0)},
\end{split}
\end{equation*}
with $r_1=\frac{r_0}{2}$. If $C_0\varepsilon_0^{p-1}<q'_0-1$, then we apply this last estimate with \eqref{q=1 w}, obtaining the following, for $r_2=\frac{r_1}{2}$:
\begin{equation*}
\begin{split}
\|w_{r_2}(\tau)\|_{L^\infty(\R^N)}\leq& C\varepsilon_0(1-\tau)^{-C_0\varepsilon_0^{p-1}}+C\varepsilon_0(1-\tau)^{\min(1-q_0',0)}.
\end{split}
\end{equation*}
Reiterating the same arguments until obtaining the first $n_0\geq 1$ such that $q'_0-\frac{n_0}{2}\leq C_0\varepsilon_0^{p-1}$, which implies

\begin{equation*}
\begin{split}
\|w_{r_{n}}(\tau)\|_{L^\infty(\R^N)}\leq& C\varepsilon_0(1-\tau)^{-C_0\varepsilon_0^{p-1}}, \forall n\geq n_0,
\end{split}
\end{equation*}
and where $r_n=\frac{1}{2^n}$. Thus,
$$r_n=\frac{1}{2^{n}}, q_n=q_0,q'_n=C_0\varepsilon_0^{p-1}, \forall n\geq n_0,$$
for small enough $\varepsilon_0$. Applying \eqref{q=1 v} for again small enough $\varepsilon_0$, we have
\begin{equation}
\begin{split}
\| v_{r_{n_0}}(\tau)\|_{L^\infty(\R^N)}\leq C\varepsilon_0(1-\tau)^{-C\varepsilon_0^{p-1}}+C\varepsilon_0 (1-\tau)^{\min(\frac{1}{2}-q_{n_0},0)}.
\end{split}
\end{equation}
Again, with the same argument as before, there exists $n_1\geq n_0$ such that 
\begin{equation}\label{sequence r,q,q'}
r_n=\frac{1}{2^{n}}, q_n=q'_n=C\varepsilon_0^{p-1}, \forall n\geq n_1.
\end{equation}

\textbf{Step 2: If} $\mathbf{\bar \gamma>1}$. we consider again  $r_0=1$, $q_0=\frac{1}{p-1}$ and $q'_0=\frac{1}{p-1}+\frac{1}{2}$. By iteration again, one may see, from \eqref{q=1 w}, that

\begin{equation*}
\begin{split}
\|w_{r_n}(\tau)\|_{L^\infty(\R^N)}\leq& C\varepsilon_0(1-\tau)^{-C\varepsilon_0^{p-1}}+C\varepsilon_0(1-\tau)^{\min(\frac{1}{2}-\bar \gamma q_n',0)}, \forall q_n'\geq 0,
\end{split}
\end{equation*}
where the sequence $(r_n, q_n,q'_n)_{n\in \N}$ is given by 
the following, for all $n\in \N$:

\begin{equation}\label{rn, qn, q'n}
\begin{split}
&r_{n+1}=\frac{1}{2^{n+1}},\\
& q_{n+1}=q_0,\\
&q'_{n+1}=f(q'_n), q'_0=\frac{1}{p-1}+\frac{1}{2},
\end{split}
\end{equation}
where $f:x\rightarrow\max(\bar \gamma x-\frac{1}{2},C\varepsilon_0^{p-1})$. One may see, for all $x\in [\hat{q'},\bar{q'}]$ where $\hat{q'}=C\varepsilon_0^{p-1}$ and $\bar{q'}=\frac{1}{2(\bar \gamma-1)}$, that we have $f(x)\leq x$. Since $q'_0\in [\hat{q'},\bar{q'})$, then $(q'_n)_{n\in \N}$ is a decreasing sequence. In addition, for all $n\in \N$, we have $q'_n\in [\hat{q'},\bar{q'})$. To prove such thing, one may see, for the upper bound, that we just need to use the decay of the sequence. For the lower bound, it can be proved with a simple induction reasoning. Therefore, the sequence converges to its lower bound $\hat{q'}$.

But we want more: The sequence is in fact stationary at $\hat{q'}$. In other words, there exists $n_0\in \N$ such that, for all $n\geq n_0$, we have $q'_{n+1}=C\varepsilon_0^{p-1}$. To prove this, we have from \eqref{rn, qn, q'n} that it is equivalent to proving the following:
\begin{equation}\label{q'_n}
q'_n\leq \frac{C\varepsilon_0^{p-1}}{\bar \gamma}+\frac{1}{2\bar{\gamma}},\ \forall n\geq n_0.
\end{equation}
Since $q'_n\underset{n\rightarrow +\infty}{\longrightarrow} C\varepsilon_0^{p-1}$, then for small enough $\varepsilon_0$, there exists $n_0\in \N$ such that, for all $n\geq n_0, q'_n\leq \frac{1}{2\bar{\gamma}}$, which implies 
\eqref{q'_n}. Hence,
\begin{equation*}
r_n=\frac{1}{2^{n}}, q_n=q_0,q'_n=C\varepsilon_0^{p-1}, \forall n\geq n_0.
\end{equation*}
With the same arguments as in Step 1, there exists $n_1\geq n_0$ such that
\begin{equation*}
r_n=\frac{1}{2^{n}}, q_n=q'_n=C\varepsilon_0^{p-1}, \forall n\geq n_1.
\end{equation*}

\textbf{Step 3:} From \eqref{q<1 v} and \eqref{q<1 w} for small enough $\varepsilon_0$ and $\mu$, we obtain
$$r_n=\frac{1}{2^{n}}, q_n=q'_n=0, \forall n\geq n_1+1.$$ 
In other words
\begin{equation}\label{conlusion n1}
\|v_{r_{n_1+2}}(\tau)\|_{L^\infty(\R^N)}+\|w_{r_{n_1+2}}(\tau)\|_{L^\infty(\R^N)}\leq C\varepsilon_0, \forall \tau \in [0,1),
\end{equation}
which concludes the proof.
\end{proof}

\section{Conclusion of the proof of Theorem \ref{Thm: ch1}'}\label{Conclusion Thm 2}
In this section, we gather the informations from the previous sections to conclude the proof of Theorem \ref{Thm: ch1}'. We proceed in two subsections, with the first dedicated to the general case, while the second one is concerned with further refinement when $N=1$.
\subsection{Theorem's proof in general dimension}
From \eqref{def: v}, we consider $\eta(\xi,\tau)=v_1(\xi,\tau)-v_2(\xi,\tau)$. We consider $\varepsilon_0>0$, and from Corollary \ref{cor bound v, nabla v}, $K_0>0$ and $r_0>0$. We know, from \eqref{bound v}, that for all $|\xi|\leq\beta (|\log(T_1-t(x))|)^{\frac{1}{4}}-1,\tau \in [0,1)$, $|\eta(\xi,\tau)|<C\varepsilon_0$, for some $\varepsilon_0>0$, and verifies the following equation (see \eqref{eq vi}):
\begin{equation}\label{eq eta}
\partial_\tau \eta=\Delta \eta+\alpha_1 \eta+\alpha_2,
\end{equation}
where 

\begin{equation*}
\begin{split}
\alpha_1=&\frac{|v_1|^{p-1}v_1-|v_2|^{p-1}v_2}{v_1-v_2},\\
\alpha_2=&(T_1-t(x))^{\frac{p}{p-1}}\left[h_1\left((T_1-t(x))^{-\frac{1}{p-1}}v_1,(T_1-t(x))^{-\frac{p+1}{2(p-1)}}\nabla v_1\right)\right.\\
&\left. -h_2\left((T_1-t(x))^{-\frac{1}{p-1}}v_2,(T_1-t(x))^{-\frac{p+1}{2(p-1)}}\nabla v_2\right)\right].
\end{split}
\end{equation*}
We consider the following cut-off function:
\begin{equation}\label{chi a}
\chi=\chi_0\left(\frac{\xi}{a|\log(T_1-t(x))|^{\theta}}\right),
\end{equation}
where $\chi_0\in C^\infty$, $\chi_0\equiv 1$ on $B(0,1)$ and $\chi_0 \equiv 0$ on $\R^N \backslash B(0,2)$ and $a<\beta$ with $a (|\log(T_1-t(x))|)^{\theta}< \beta (|\log(T_1-t(x))|)^{\frac{1}{4}}-1$ and $0<\theta\leq\frac{1}{4}$, which is going to be determined afterward. We also define 
$\bar{\eta}=\eta \chi$. From \eqref{eq eta}, we have that $\bar\eta$ verifies the following equation:
\begin{equation}\label{eq: eta bar}
\begin{split}
\partial_\tau \bar{\eta}&=\chi \Delta \eta+\alpha_1\bar{\eta}+\chi\alpha_2\\
&= \Delta \bar{\eta}-2\nabla(\eta\nabla \chi)+\eta \Delta \chi+\alpha_1\bar{\eta}+\chi\alpha_2\\
&= \Delta \bar{\eta}-\frac{2\nabla(\eta\nabla \chi_0)}{a|\log (T_1-t(x))|^{\theta}}+\frac{\eta \Delta \chi_0}{a^2|\log (T_1-t(x))|^{2\theta}}+\alpha_1\bar{\eta}+\chi\alpha_2,\\
\end{split}
\end{equation}
where we used, for the last equality, the fact that $\nabla \chi=\frac{\nabla \chi_0}{a|\log (T_1-t(x))|^{\theta}}$ and $\Delta \chi=\frac{\Delta \chi_0}{a^2|\log (T_1-t(x))|^{2\theta}}$. Therefore, $\bar \eta$ can be written as

\begin{equation*}
\begin{split}
\bar{\eta}(\tau)=&S(\tau) \bar{\eta}(0)-\frac{2}{a|\log (T_1-t(x))|^{\theta}}\int^\tau_0S(\tau-s)\nabla(\eta\nabla \chi_0)(s)\ ds\\
&+\frac{1}{a^2|\log (T_1-t(x))|^{2\theta}}\int^\tau_0S(\tau-s)\eta(s) \Delta \chi_0\ ds\\
&+\int^\tau_0 S(\tau-s)\alpha(s)\bar{\eta}(s)\ ds+\int^\tau_0 S(\tau-s)\chi\alpha_2(s)\ ds,
\end{split}
\end{equation*}
where $S$ is the heat semigroup kernel. Then, from \eqref{heat kernel prp}, we have
\begin{equation*}
\begin{split}
\|\bar{\eta}(\tau)\|_{L^\infty(\R^N)}\leq&\|\bar{\eta}(0)\|_{L^\infty(\R^N)}+\frac{C}{|\log (T_1-t(x))|^{\theta}}\int^\tau_0(\tau-s)^{-\frac{1}{2}}\|\eta(s)\nabla \chi_0\|_{L^\infty(\R^N)}\ ds\\
&+\frac{C}{|\log (T_1-t(x))|^{2\theta}}\int^\tau_0\|\eta (s) \Delta \chi_0\|_{L^\infty(\R^N)}\ ds+\int^\tau_0 \|\alpha_1(s)\bar{\eta}(s)\|_{L^\infty(\R^N)}\ ds\\
&+C(T_1-t(x))^{\frac{p-\gamma}{p-1}}\int^\tau_0 \|\chi|v_1|^{\gamma}(s)+\chi|v_2|^{\gamma}(s)\|_{L^\infty(\R^N)}\ ds\\
&+C(T_1-t(x))^{\frac{2p-\bar \gamma(p+1)}{p-1}}\int^\tau_0\|\chi|\nabla v_1|^{\bar\gamma}(s)+\chi|\nabla v_2|^{\bar\gamma}(s)\|_{L^\infty(\R^N)}ds\\
&+C(T_1-t(x))^{\frac{p}{p-1}}\chi\int^\tau_0ds.
\end{split}
\end{equation*}
Thus, 
\begin{equation*}
\begin{split}
\|\bar{\eta}(\tau)\|_{L^\infty(\R^N)}&\leq\|\bar{\eta}(0)\|_{L^\infty(\R^N)}+\frac{C}{|\log (T_1-t(x))|^{\theta}}\int^\tau_0(\tau-s)^{-\frac{1}{2}}\|\eta (s)\nabla \chi_0\|_{L^\infty(\R^N)}\ ds\\
&\ +\frac{C}{|\log (T_1-t(x))|^{2\theta}}\int^\tau_0\|\eta (s)\Delta \chi_0\|_{L^\infty(\R^N)}\ ds+\int^\tau_0 \|\alpha_1(s)\bar{\eta}(s)\|_{L^\infty(\R^N)}\ ds\\
&\ +C(T_1-t(x))^{\nu }.
\end{split}
\end{equation*}
We remind that $\nu=\frac{p-\gamma_0}{p-1}$ and $\gamma_0=\max\{\gamma,\overline{\gamma}(p+1)-p\}$. There exists $v_0\in (v_1,v_2)$ such that $\alpha_1=p|v_0|^{p-1}$. Hence, from \eqref{bound v}, there exists a constant $C>0$ such that $\|\alpha_1\|_{L^\infty(\R^N)}\leq C$. Also, there exists a constant $C>0$ such that $\|\nabla \chi_0\|_{L^\infty(\R^N)}\leq C$, $\|\Delta \chi_0\|_{L^\infty(\R^N)}\leq C$ and $\|\eta\|_{L^\infty(\R^N)}\leq C$. Then,

\begin{equation*}
\begin{split}
\|\bar{\eta}(\tau)\|_{L^\infty(\R^N)}\leq&\|\bar{\eta}(0)\|_{L^\infty(\R^N)}+\frac{C}{|\log (T_1-t(x))|^{\theta}}\int^\tau_0(\tau-s)^{-\frac{1}{2}}\ ds+\frac{C}{|\log (T_1-t(x))|^{2\theta}}\\
&+C\int^\tau_0 \|\bar{\eta}(s)\|_{L^\infty(\R^N)}\ ds+C(T_1-t(x))^{\nu }.
\end{split}
\end{equation*}
We obtain, for small $|x|$,
\begin{equation}\label{eta 0 left}
\begin{split}
\|\bar{\eta}(\tau)\|_{L^\infty(\R^N)}&\leq\|\bar{\eta}(0)\|_{L^\infty(\R^N)}+\frac{C}{|\log (T_1-t(x))|^{\theta}}+C\int^\tau_0 \|\bar{\eta}(s)\|_{L^\infty(\R^N)}\ ds.
\end{split}
\end{equation}

For $\tau=0$, we have $t=t(x)$. Therefore, from \eqref{def t(x)}, we have $|x|= \frac{K_0}{2}\sqrt{(T_1-t)|\log T_1-t|}$. Thus, from \eqref{inequality: error2}, we have

\begin{equation*}
\underset{\xi\leq a |\log(T_1-t(x))|^{\theta}}{\sup}|\eta(\xi,0)|=\underset{\xi\leq a |\log(T_1-t(x))|^{\theta}}{\sup}\left|v_1(\xi,0)-v_2(\xi,0)\right|\leq \frac{C}{|\log(T_1-t(x))|}.
\end{equation*}
Applying this to \eqref{eta 0 left}, we obtain
\begin{equation}\label{final controll equation eta}
\begin{split}
\|\bar{\eta}(\tau)\|_{L^\infty(\R^N)}&\leq \frac{C}{|\log(T_1-t(x))|}+\frac{C}{|\log (T_1-t(x))|^{\theta}}+C\int^\tau_0 \|\bar{\eta}(s)\|_{L^\infty(\R^N)}\ ds.
\end{split}
\end{equation}
Using a Gronwall's argument on \eqref{final controll equation eta} and knowing that $\tau <1$, we obtain

\begin{equation*}
\begin{split}
\|\bar{\eta}(\tau)\|_{L^\infty(\R^N)}&\leq \frac{C}{|\log(T_1-t(x))|}+\frac{C}{|\log (T_1-t(x))|^{\theta}}.
\end{split}
\end{equation*}
Therefore, we just need to take $\theta=\frac{1}{4}$ since $\theta\leq \frac{1}{4}$, and we obtain, for all $\xi \leq a (|\log(T_1-t(x))|)^{\frac{1}{4}}$ and $\tau \in [0,1)$,
\begin{equation*}
\begin{split}
|u_1(\xi,\tau)-\overset{\sim}{u}_2(\xi,\tau)|\leq \frac{C(T_1-t(x))^{-\frac{1}{p-1}}}{|\log(T_1-t(x))|^{\frac{1}{4}}}.
\end{split}
\end{equation*}
We can see, when $|x|\rightarrow 0$, that we have the following behavior:
\begin{equation}\label{equiv t(x)}
\log(T_1-t(x))\sim 2\log |x|\text{ and } T_1-t(x)\sim \frac{2|x|^2}{K_0^2|\log|x||}.
\end{equation}
Hence, there exists $\epsilon\in  (0; K_0 \sqrt{T_1 \log T_1})$ and small $\delta_0$ such that we have, for all\\ $\epsilon\geq |x|\geq K_0 \sqrt{(T_1-t )|\log (T_1-t)|}$ and $t\in [T_1-\delta_0,T_1)$,

$$|u_1(x,t)-\overset{\sim}{u}_2(x,t)| \leq \frac{C|x|^{-\frac{2}{p-1}}}{|\log|x||^{\frac{1}{4}-\frac{1}{p-1}}}.$$
Together with \eqref{inequality: error2}, we may conclude \eqref{thm estimation dim N}.

\subsection{Theorem's proof in dimension 1}
We have the following parabolic regularity result:
\begin{lemma}[Parabolic regularity]\label{bound z}
Assume that, for all $ |\xi|\leq 4B_1$ and $\tau\in  [0;\tau_*)$, $z$ satisfies
\begin{equation}\label{eq z}
\begin{split}
&\partial_\tau z\leq \Delta z+\lambda_1 z+\lambda_2,\\
&z(\xi,0)\leq z_0,\ z(\xi,\tau)\leq B_2,
\end{split}
\end{equation}
where $\tau_*\leq 1$. Then, for all $|\xi|\leq B_1$ and $\tau\in  [0;\tau_*),$
$$z(\xi, \tau)\leq e^{\lambda_1\tau}(z_0+\lambda_2+CB_2e^{-\frac{B_1^2}{4}}).$$
\end{lemma}
\begin{proof}
See Appendix C in \cite{MZ1998}.
\end{proof}

We define from \eqref{inequality: error3} the following:
\begin{equation*}
\begin{split}
&v_1(\xi,\tau)=(T_1-t(x))^{\frac{1}{p-1}}u_1(x+\xi\sqrt{T_1-t(x)},t(x)+\tau (T_1-t(x))),\\
&\overline{v}_2(\xi,\tau)=(T_1-t(x))^{\frac{1}{p-1}}\overline{u}_2(x+\xi\sqrt{T_1-t(x)},t(x)+\tau (T_1-t(x))),\\
&\overset{\sim}{\eta}(\xi,\tau)=v_1(\xi,\tau)-\overline{v}_2(\xi,\tau),
\end{split}
\end{equation*}
where $\overline{u}_2$ is given in \eqref{inequality: error3}. One may see that $\overline{v}_2$ and $\overset{\sim}{\eta}$ satisfy \eqref{bound v} and \eqref{eq eta} respectively. Together with similar computations as for \eqref{expression alpha_1} and \eqref{alpha 2 bound computations}, we obtain that $\overset{\sim}{\eta}$ satisfies \eqref{eq z} with 

\begin{equation*}
\begin{split}
&B_1=\frac{1}{4}\left(\beta (|\log(T_1-t(x))|)^{\frac{1}{4}}-1\right),\ \lambda_1=Cp\varepsilon_0^{p-1},\ B_2=C\varepsilon,\ \lambda_2=C(T_1-t(x))^{\nu},\\
&z_0=C(K_0)\max\left(\frac{(T_1-t(x))^{\frac{1}{2}}}{|\log(T_1-t(x))|^{\frac{3}{2}}}, \frac{(T_1-t(x))^{\nu}}{|\log(T-t(x))|^{2-\nu}}\exp\left(C\sqrt{-\log(T-t(x))}\right)\right).
\end{split}
\end{equation*}
Therefore, by Lemma \ref{bound z} and for small $|x|$, we have 

\begin{equation*}
\begin{split}
|\overset{\sim}{\eta}(\xi,\tau)|=|v_1(x,t)-\overline{v}&_2(x,t)|\leq \ C(K_0)\max\left(\frac{(T_1-t(x))^{\frac{1}{2}-\frac{1}{p-1}}}{|\log(T_1-t(x))|^{\frac{3}{2}}},\right.\\
&\ \left. \frac{(T_1-t(x))^{\nu-\frac{1}{p-1}}}{|\log(T-t(x))|^{2-\nu}}\exp\left(C\sqrt{-\log(T-t(x))}\right)\right)+C(T_1-t(x))^{\nu-\frac{1}{p-1}}.
\end{split}
\end{equation*}
Together with \eqref{equiv t(x)}, we obtain

\begin{equation}
\begin{split}
|u_1(x,t)-\overline{u}_2(x,t)|\leq & C(K_0)\max\left(\frac{|x|^{1-\frac{2}{p-1}}}{|\log |x||^{2-\frac{1}{p-1}}}, \frac{|x|^{2\nu-\frac{2}{p-1}}}{|\log |x||^{2-\frac{1}{p-1}}}\exp\left(C\sqrt{-\log |x|}\right)\right)\\
&+C\frac{|x|^{2\nu-\frac{2}{p-1}}}{|\log|x||^{\nu-\frac{1}{p-1}}}.
\end{split}
\end{equation}
Thus,
$$|u_1(x,t)-\overline{u}_2(x,t)|\leq C(K_0)\max\left(\frac{|x|^{\min(1,2\nu)-\frac{2}{p-1}}}{|\log |x||^{\min(2,\nu)-\frac{1}{p-1}}}, \frac{|x|^{2\nu-\frac{2}{p-1}}}{|\log |x||^{\min(2,\nu)-\frac{1}{p-1}}}\exp\left(C\sqrt{-2\log |x|}\right)\right).$$
Together with \eqref{inequality: error3}, we conclude \eqref{thm estimation dim 1}. Allowing us to close the proof of Theorem \ref{Thm: ch1}'.

\section{Final blow-up profile}
This section is devoted to the proof of Theorem \ref{Thm 2 g}, where we derive the final blow-up profile and prove the single point blow-up property. 

\begin{remark}
We will do the proof in the one-dimensional case to simplify the notations. One will see that it can be easily generalized to the multi-dimensional one. 
\end{remark}
\begin{proof}[Proof of Theorem \ref{Thm 2 g}]
Let us consider $u$ solution of \eqref{eq:1} such that $u\in S_{a,T,h}$. From equation \eqref{eq:1}, we may assume that $a=0$. Theorem \ref{eq:1} will follow if we prove the following two facts:
\begin{itemize}
\item The origin is the only blow-up point in $B(0,\epsilon)$ for some $\epsilon>0$, and $\ref{Thm 2 ii}$ holds in that same ball.
\item For any $|x|\geq \epsilon$, $x$ is not a blow-up point and $u(x,t)$ has a limit $u_*(x)$ as $t\rightarrow T$.
\end{itemize}
We proceed in two steps to prove these two facts.
\bigskip

\textbf{Step 1 (Single point blow-up in $B(0,\epsilon)$ and final profile):} One may see that the solution of the following equation: 
\begin{equation*}
\begin{split}
&v'_{K_0}=|v_{K_0}|^{p-1}v_{K_0},\\
&v_{K_0}(0)=f\left(\frac{K_0}{2}\right),
\end{split}
\end{equation*}
is given by the following: 
\begin{equation}\label{v_K_0}
v_{K_0}(\tau)=\left(f\left(\frac{K_0}{2}\right)^{1-p}-(p-1)\tau\right)^{-\frac{1}{p-1}}.
\end{equation}

We define $\eta(\xi,\tau)=v(\xi,\tau)-v_{K_0}(\tau)$, where $v$ is defined the same as $v_1$ in \eqref{def: v} with $T_1=T$ and $a=0$. We know from \eqref{bound v} and \eqref{v_K_0} that, for all $|\xi|\leq\beta (|\log(T_1-t(x))|)^{\frac{1}{4}}-1$ and $\tau \in [0,1)$, $|\eta(\xi,\tau)|<+\infty$ and $\eta$ verifies the following equation:
$$\partial_\tau \eta=\Delta \eta+\alpha_1 \eta+\alpha_2,$$
where 
\begin{equation*}
\begin{split}
&\alpha_1=\frac{|v|^{p-1}v-|v_{K_0}|^{p-1}v_{K_0}}{v-v_{K_0}},\\
&\alpha_2=(T_1-t(x))^{\frac{p}{p-1}}h\left((T_1-t(x))^{-\frac{1}{p-1}}v,(T-t(x))^{-\frac{p+1}{2(p-1)}}\nabla v\right).
\end{split}
\end{equation*}
Hence, we define again $\bar{\eta}=\eta \chi$, where $\chi$ is the same as in \eqref{chi a}. Thus, $\bar\eta$ verifies the following equation:
\begin{equation*}
\begin{split}
\partial_\tau \bar{\eta}=\chi \Delta \bar{\eta}-\frac{2\nabla(\eta\nabla \chi_0)}{a|\log (T-t(x))|^{\theta}}+\frac{\eta \Delta \chi_0}{a^2|\log (T-t(x))|^{2\theta}}+\alpha_1\bar{\eta}+\chi \alpha_2.\\
\end{split}
\end{equation*}
Therefore, it can be written as
\begin{equation*}
\begin{split}
\bar{\eta}(\tau)&=S(\tau) \bar{\eta}(0)-\frac{2}{a|\log (T_1-t(x))|^{\theta}}\int^\tau_0S(\tau-s)\nabla(\eta(s)\nabla \chi_0)ds\\
&+\frac{1}{a^2|\log (T_1-t(x))|^{2\theta}}\int^\tau_0S(\tau-s)\eta(s) \Delta \chi_0ds+\int^\tau_0 S(\tau-s)\alpha_1(s)\bar{\eta}(s)ds\\
&+\int^\tau_0 S(\tau-s)\chi\alpha_1(s)ds,
\end{split}
\end{equation*}
where $S$ is again the heat semigroup kernel. With the same computations as done in the previous section, we have

\begin{equation}\label{control bar eta}
\begin{split}
\|\bar{\eta}(\tau)\|_{L^\infty(\R^N)}&\leq\|\bar{\eta}(0)\|_{L^\infty(\R^N)}+\frac{C}{|\log (T_1-t(x))|^{\theta}}\int^\tau_0(\tau-s)^{-\frac{1}{2}}ds+\frac{C}{|\log (T_1-t(x))|^{2\theta}}\int^\tau_0ds\\
&+C\int^\tau_0 \|\bar{\eta}(s)\|_{L^\infty(\R^N)}ds+C(T_1-t(x))^{\nu}.
\end{split}
\end{equation}
Thus, for small $|x|$, we have
\begin{equation}\label{eta 0 left 2}
\begin{split}
\|\bar{\eta}\|_{L^\infty(\R^N)}(\tau)&\leq\|\bar{\eta}(0)\|_{L^\infty(\R^N)}+\frac{C}{|\log (T_1-t(x))|^{\theta}}+C\int^\tau_0 \|\bar{\eta}(s)\|_{L^\infty(\R^N)}ds.
\end{split}
\end{equation}
We know that $f$ is a Lipschitz function. Therefore,
\begin{equation}\label{f z-K_0}
\left|f(z)-f\left(\frac{K_0}{2}\right)\right|\leq C\left||z|-\frac{K_0}{2}\right|,
\end{equation}
where 
$$z=\frac{x+\xi\sqrt{T_1-t(x)}}{\sqrt{(T_1-t)|\log(T_1-t)|}}.$$
For $\tau =0$, we have $t=t(x)$, and for $x\geq 0$, we have, from \eqref{def t(x)},
\begin{equation*}
\begin{split}
|z|&\leq\frac{K_0}{2}\frac{\sqrt{(T_1-t(x))|\log(T_1-t(x))|}}{\sqrt{(T_1-t)|\log(T_1-t)|}}+\frac{|\xi|\sqrt{T_1-t(x)}}{\sqrt{(T_1-t)|\log(T_1-t)|}}\\
&=\frac{K_0}{2}+\frac{|\xi|}{\sqrt{|\log(T_1-t(x))|}}.\\
\end{split}
\end{equation*}
Since $|\xi|\leq \beta_1|\log(T_1-t(x))|^\theta$, then we get
\begin{equation*}
||z|-\frac{K_0}{2}|\leq \frac{1}{a|\log(T_1-t(x))|^{\frac{1}{2}-\theta}}.
\end{equation*}
Together with \eqref{profile v} and \eqref{f z-K_0}, we have 
\begin{equation*}
\underset{\xi\leq a |\log(T_1-t(x))|^{\theta}}{\sup}\left|v(\xi,0)-f(z)\right|\leq \frac{C}{|\log(T_1-t(x))|^{\frac{1}{2}-\theta}}.
\end{equation*}
Thus,
\begin{equation}\label{profile v(0)}
\underset{\xi\leq a |\log(T_1-t(x))|^{\theta}}{\sup}|\eta(\xi,0)|\leq \frac{C}{|\log(T_1-t(x))|^{\frac{1}{2}-\theta}}.
\end{equation}
Applying this to \eqref{eta 0 left 2}, we obtain
\begin{equation*}
\begin{split}
\|\bar{\eta}(\tau)\|_{L^\infty(\R^N)}&\leq \frac{C}{|\log(T_1-t(x))|^{\frac{1}{2}-\theta}}+\frac{C}{|\log (T_1-t(x))|^{\theta}}+C\int^\tau_0 \|\bar{\eta}(s)\|_{L^\infty(\R^N)}ds.
\end{split}
\end{equation*}
Using a Gronwall's argument again and by taking the optimal $\theta$ which is $\theta=\frac{1}{4}$, we get 
$$\|\bar{\eta}(\tau)\|_{L^\infty(\R^N)} \leq \frac{C}{|\log(T_1-t(x))|^{\frac{1}{4}}}.$$
Therefore, for $\xi\leq \beta_1|\log(T_1-t(x))|^{\frac{1}{4}}$ and $\tau \in [0,1)$,
$$|v(\xi,\tau)-v_{K_0}(\tau)|\leq \frac{C}{|\log(T_1-t(x))|^{\frac{1}{4}}}.$$
Then, for $\epsilon\in  (0; \frac{K_0}{2} \sqrt{T_1 \log T_1})$ and small $\delta_0$, for all $\epsilon\geq |x|\geq K_0 \sqrt{T_1-t \log (T_1-t)}$ and $t\in [T-\delta_0,T)$, we have
$$\left|u(x,t)-(T_1-t(x))^{-\frac{1}{p-1}}v_{K_0}\left(\frac{t-t(x)}{T_1-t(x)}\right)\right|\leq  C\frac{(T_1-t(x))^{-\frac{1}{p-1}}}{|\log(T_1-t(x))|^{\frac{1}{4}}}.$$
We define , for all $x\in \R\backslash \{a\}$ non-blow-up points, the following:
$$u_*(x)=\underset{t\rightarrow T}{\lim}u(x,t).$$
Then, for some  $\epsilon\in  (0; K_0 \sqrt{T \log T})$, we have for all $\epsilon\geq |x|\geq K_0 \sqrt{T-t \log (T-t)},$

$$|u_*(x)-(T-t(x))^{-\frac{1}{p-1}}v_{K_0}(1)|\leq  C\frac{(T_1-t(x))^{-\frac{1}{p-1}}}{|\log(T_1-t(x))|^{\frac{1}{4}}}.$$
Thus, from \eqref{equiv t(x)} with small enough $\delta_0$, we have for all $\epsilon\geq |x|\geq K_0 \sqrt{(T-t )\log (T-t)}$ and $t\in [T-\delta_0,T),$
\begin{equation}\label{final profile proof}
\left|u_*(x)-\left(\frac{2}{K_0^2}\right)^{-\frac{1}{p-1}}\frac{|x|^{-\frac{2}{p-1}}}{|\log|x||^{-\frac{1}{p-1}}}\left(f\left(\frac{K_0}{2}\right)^{1-p}-(p-1)\right)^{-\frac{1}{p-1}}\right|\leq  C(K_0) \frac {|\log|x-a||^{\frac{1}{p-1} - \frac{1}{4}}} {|x-a|^{\frac{2}{p-1}}},
\end{equation}
which gives us \eqref{Thm 2}. 

\textbf{Step 2 (No blow-up outside of $\mathbf{B(0,\epsilon)}$):} We may assume $T$ is as small as needed. Otherwise, we just have to consider in \eqref{eq:1}, $u_0(.)=u(.,T-\varepsilon)$, where $\varepsilon>0$ is small enough.

\textbf{Case of $\mathbf{\epsilon < |x|<\frac{K_0}{2}\sqrt{T|\log T|}}$:} We have from \eqref{z > K_0 beta}
$$|z|\geq \frac{K_0}{4}.$$
Together, with \eqref{profile v}, we obtain \begin{equation*}
(1-\tau)^{\frac{1}{p-1}}|v(\xi,\tau)|\leq \varepsilon_0, \forall |\xi|\leq\beta (|\log(T-t(x))|)^{\frac{1}{4}}, \forall \tau \in [0,1).
\end{equation*}
Easy similar computations yield to 
\begin{equation*}
|\nabla v(\xi,\tau)|\leq \varepsilon_0 (1-\tau)^{-\frac{1}{p-1}-\frac{1}{2}}, \forall \xi\in \R^N, \forall \tau \in [0,1).
\end{equation*}
With Proposition \ref{prp bound v}, we have for $|\xi|\leq\beta (|\log(T-t(x))|)^{\frac{1}{4}},\tau\in[0,1)$, 
$$|v(\xi,\tau)|\leq C.$$
Since $t(x)<T-\varepsilon$, for some $\varepsilon>0$, then 
\begin{equation}\label{control eta x>epsilon}
u_*(x)\leq C \text{ for all }\epsilon\leq |x|<\frac{K_0}{2}\sqrt{T|\log T|}.
\end{equation}

\textbf{Case of $\mathbf{|x| \geq \frac{K_0}{2}\sqrt{T|\log T|}}$:} We define for all $\tau(t)=\frac{t}{T}\in [0,1)$ and $|\xi|\leq \beta |\log T|^{\frac{1}{4}}$,
$$v(\xi,\tau)=T^{\frac{1}{p-1}}u(x+\xi\sqrt{T_1},\tau T).$$
Then, for $z=\frac{x+\xi\sqrt{T}}{\sqrt{(T-t)|\log(T-t)|}}$, we have
\begin{equation}\label{z K0}
\begin{split}
|z|&=\frac{\sqrt{T|\log T|}}{\sqrt{(T-t)|\log(T-t)|}}\left|\frac{x}{\sqrt{T|\log T|}}+\frac{\xi}{\sqrt{|\log T|}}\right|\\
&\geq \frac{\sqrt{T|\log T|}}{\sqrt{(T-t)|\log(T-t)|}}\left(\frac{K_0}{2}-\frac{\beta}{|\log T|^{\frac{1}{4}}}\right)\\
&\geq \frac{K_0}{4}.
\end{split}
\end{equation}
From \eqref{behaviour1}, we have
\begin{equation*}
\begin{split}
\underset{\xi\in \R^N}{\sup}\left|(1-\tau)^{\frac{1}{p-1}}v(\xi,\tau)-f\left(z\right)\right|&\leq \frac{C}{\sqrt{|\log(T-t)|}}\\
&\leq \frac{C}{\sqrt{|\log T|}}.
\end{split}
\end{equation*}
Together with \eqref{z K0}, we have for all $\tau(x,t)=\frac{t}{T}\in [0,1)$ and $|\xi|\leq \beta |\log T|^{\frac{1}{4}}$,
\begin{equation}\label{rough v 2}
(1-\tau)^{\frac{1}{p-1}}|v(\xi,\tau)|\leq \varepsilon_0.
\end{equation}
Again from \eqref{behaviour1}, we have
\begin{equation*}
\underset{x'\in \R^N}{\sup}\left|(1-\tau)^{\frac{1}{p-1}+\frac{1}{2}}\nabla v(\xi,\tau)-\frac{1}{\sqrt{|\log T|}}\nabla f\left(z\right)\right|\leq \frac{C}{\sqrt{|\log T|}},
\end{equation*}
which gives us 
\begin{equation*}
|\nabla v(\xi,\tau)|\leq \varepsilon_0 (1-\tau)^{-\frac{1}{p-1}-\frac{1}{2}}.
\end{equation*}
Together with \eqref{rough v 2} and Proposition \ref{prp bound v}, we obtain for $|\xi|\leq\beta (|\log(T-t(x))|)^{\frac{1}{4}},\tau\in[0,1)$, 
$$|v(\xi,\tau)|\leq C.$$
Again since $t(x)<T-\varepsilon$, then 

$$u_*(x)\leq C \text{ for all }|x| \geq \frac{K_0}{2}\sqrt{T|\log T|}.$$
Together with \eqref{final profile proof} and \eqref{control eta x>epsilon}, we conclude that $0$ is the only blow-up point and that $u_*(x)$ exists for all $x\in \R\backslash \{a\}$, which concludes the proof of Theorem \ref{Thm 2 g}.
\end{proof}

\appendix

\section{Proof of Lemma \ref{lemma:FL1999}}\label{Lemma FK1999}
The proof of Lemma \ref{lemma:FL1999} is as in \cite{FK1992} and \cite{FL1993}. Of course, those papers consider the case where $h\equiv 0$. We will consider the general case here. In particular, our contribution is visible in Lemmas \ref{lemma eq v+,v-,v null} and \ref{lemma eq neutral modes}. We will recall the main ideas of the proof (see \cite[Section 4 and 5]{FK1992} and \cite[Proposition 4.1]{FL1993}), only pointing out the parts where changes occur in our case. We assume $N=1$ and define $v=w_2-\kappa$. Since we have $|w_2(y,s)|\leq M$, for all $s\geq s_0$ and $y\in\R$, from Lemma \ref{bound w}, then we have $|v(y,s)|\leq B:=M+\kappa$. From \eqref{eq1:cv}, the equation satisfied by $v$ is 
\begin{equation}\label{eq v Append}
v_s=\mathcal{L}v+G(v)+F(v,\nabla v),
\end{equation}
where $\mathcal{L}$ is given in \eqref{def op L}, $G(v)=|v+\kappa|^{p-1}(v+\kappa)-\kappa^p-\frac{p}{p-1}v$, $F(v,\nabla v)=e^{-\frac{p}{p-1}s}h(e^{\frac{1}{p-1}s}(v+\kappa),e^{\frac{p+1}{2(p-1)}s}\nabla v)$. We define 
\begin{align}\label{x y z}
&v_+=P_+(v)=P_0(v)+P_1(v),&z=\|v_+\|_{L^2_\rho},\\
&v_{null}=P_2(v),&x=\|v_{null}\|_{L^2_\rho},\\
&v_-=P_-(v)=\underset{n\geq 3}{\sum}P_nv,&y=\|v_-\|_{L^2_\rho}.
\end{align}
Then, we have
\begin{equation}\label{decomposition v}
    v=v_++v_{null}+v-.
\end{equation}
We have the following first lemma:
\begin{lemma}\label{lemma eq v+,v-,v null}
Either $v \rightarrow 0$ exponentially fast as $s \rightarrow +\infty$, or for some $b>0$ and for any $\widehat \epsilon > 0$, there is a time $s_*$ such that for all $s\geq s_*$,
\begin{equation}\label{generic behavior}
\|v_+\|_{L^2_\rho}+\|v_-\|_{L^2_\rho}+e^{-\frac{\nu}{3}s}\leq b\widehat{\epsilon}\|v_{null}\|_{L^2_\rho}.
\end{equation}
\end{lemma}
\begin{proof}
Using a Taylor expansion on \eqref{eq v Append}, we have
\begin{equation}\label{eq:v}
v_s=\mathcal{L}v+c(\phi,v)v^2+F(v,\nabla v),
\end{equation}
where $c(\phi,v)=\frac{1}{2}p(p-1)(\kappa+\phi)^{p-2}$ and $0\leq \phi \leq v$. 
Projecting it on the unstable subspace, we obtain 
$$\frac{1}{2}\frac{d}{d s}\int v^2_+\rho=\int \mathcal{L}v_+ .v_+\rho+\int P_+[c(\phi,p)v^2]v_+\rho+\int P_+[F(v,\nabla v)]v_+\rho.$$
Filippas and Kohn in \cite[pages 842 and 843]{FK1992} showed that 
$$\int \mathcal{L}v_+ .v_+\rho+\int \pi_+[c(\phi,p)v^2]v_+\rho\geq \frac{1}{2}z^2-CNz,$$
where $N=(\int v^4\rho)^{\frac{1}{2}}$. From Lemma \ref{bound w} and Proposition \ref{bound nabla w}, we have, for all $s\geq s_1$, the following:
$$|F(v,\nabla v)|\leq Ce^{-\nu s}.$$
Then, by Cauchy-Schwarz inequality, we obtain 
$$\left|\int P_+[F(v,\nabla v)]v_+\rho\right|\leq Ce^{-\nu s}\left(\int v_+^2\rho\right)^{\frac{1}{2}}=Ce^{-\nu s} z.$$
Thus,
$$\overset{.}{z}\geq \frac{1}{2}z-CN-Ce^{-\nu s}.$$
Using the same arguments with the neutral and the stable component we obtain 
\begin{equation}\label{ineq xyz}
\begin{split}
&\overset{.}{z}\geq \frac{1}{2}z-CN-Ce^{-\nu s},\\
&|\overset{.}{x}|\leq CN+Ce^{-\nu s},\\
&\overset{.}{y}\leq -\frac{1}{2}y+CN+Ce^{-\nu s}.\\
\end{split}
\end{equation}
They also showed, in \cite[(4.10)]{FK1992}, that 
\begin{equation}\label{Control N}
N\leq \epsilon (x+y+z)+\delta^{\frac{k}{2}} J,
\end{equation}
where $k$ is a positive integer and $J=(\int v^4 |y|^k\rho )^{\frac{1}{2}}$.
To estimate $J$, we multiply \eqref{eq:v} by $v^3|y|^k\rho$ and integrate to get
\begin{equation}
\frac{1}{4}\frac{d}{ds}\int v^4|y|^k\rho=\int \nabla(\rho\nabla v)v^3|y|^k+\int v^4|y|^k\rho +\int c(\phi,p)v^5|y|^k \rho+\int F(v,\nabla v)v^3|y|^k\rho.
\end{equation}
Also, from \cite[(4.15)]{FK1992}, we have that
$$\int \nabla(\rho\nabla v)v^3|y|^k+\int v^4|y|^k\rho +\int c(\phi,p)v^5|y|^k \rho\leq -\theta J^2+\epsilon '(x+y+z)J,$$
where 
$$\theta =\frac{k}{8}-1-CB-\frac{k\delta^2}{4}(k+N-2),$$
$$\epsilon '=\frac{1}{4}\epsilon \delta^{2-\frac{k}{2}}k(k+N-2),$$
 for some $\epsilon>0$ as small as needed. For the last term, we have 
$$\int F(v,\nabla v)v^3 |y|^k\rho\leq Ce^{-\nu s}\left(\int v^2 |y|^k\rho\right)^{1/2}\left(\int v^4 |y|^k\rho\right)^{1/2}\leq  C(k,N)Be^{-\nu s}J.$$
Thus,
\begin{equation}
\overset{.}{J}\leq -\theta J+\epsilon ' (x+y+z)+C(k,N)Be^{-\nu s}.
\end{equation}
We make $\theta \geq \frac{1}{2}$ by taking $k$ large enough. Then, we obtain

\begin{equation}\label{ineq J}
\overset{.}{J}\leq -\frac{1}{2}J+\epsilon ' (x+y+z)+Ce^{-\nu s}.
\end{equation}
From \eqref{ineq xyz},\eqref{Control N} and \eqref{ineq J}, we get the following system:
\begin{equation}
\begin{split}
&\overset{.}{z}\geq \frac{1}{2}z-\epsilon C (x+y+z)-\delta^{\frac{k}{2}} CJ-Ce^{-\nu s},\\
&|\overset{.}{x}|\leq \epsilon C (x+y+z)+\delta^{\frac{k}{2}} CJ+Ce^{-\nu s},\\
&\overset{.}{y}\leq -\frac{1}{2}y+\epsilon C_0 (x+y+z)+\delta^{\frac{k}{2}} C_1J+C_2(k)e^{-\nu s},\\
&\overset{.}{J}\leq -\frac{1}{2} J+\epsilon ' (x+y+z)+C_2(k)e^{-\nu s}.
\end{split}
\end{equation}
We add the last two inequalities and consider 
$$\overset{\sim}{y}=y+J+e^{-\frac{\nu}{3} s}\text{ and }\widehat{\epsilon}=\max(\epsilon C_0+\epsilon ', \delta^{\frac{k}{2}}C_1), K=\min\left(\frac{1}{2},\frac{\nu}{3}\right).$$
We can always make $\widehat{\epsilon}$ as small as we want by taking $\epsilon$ and $\delta$ small enough. Then, for $s$ large enough, we have that
$$\overset{.}{\overset{\sim}{y}}\leq -K\overset{\sim}{y}+\widehat\epsilon{\overset{\sim}{y}}+\widehat{\epsilon}(x+z)+2C_2(k)e^{-\nu s}.$$
Therefore, for $s$ large enough, we have that  $2C_2(k)e^{-\nu s}\leq \widehat \epsilon\tilde{y}$, we obtain the inequalities system 
\begin{equation}\label{system x,y,z}
\begin{split}
&\overset{.}{z}\geq \left(K-2\widehat{\epsilon}\right)z-2\widehat{\epsilon}(x+\overset{\sim}{y}),\\
&|\overset{.}{x}|\leq 2\widehat{\epsilon}(x+y+z),\\
&\overset{.}{\overset{\sim}{y}}\leq -\left(K-\widehat\epsilon\right)\overset{\sim}{y}+2\widehat{\epsilon}(x+z).
\end{split}
\end{equation}
Now, we recall the following lemma:
\begin{lemma}[Filippas and Kohn, \cite{FK1992}]\label{lemma x,y,z}
Let $x(s ) ,y (s ) , z ( s )$ be absolutely continuous, real valued functions which are non-negative and satisfy:
\begin{equation}
\begin{split}
&\overset{.}{z}\geq c_0z-\widehat\epsilon(x+y),\\
&|\overset{.}{x}|\leq \widehat\epsilon(x+y+z),\\
&\overset{.}{y}\leq -c_0 y+\widehat\epsilon(x+z),\\
&x,y,z\rightarrow 0 \text{ as } s\rightarrow \infty,
\end{split}
\end{equation}
where $c_0$ is any positive constant and $\widehat\epsilon$ is a sufficiently small positive constant. Then, either $x, y, z\rightarrow 0$ exponentially fast, or there exists a time $s_*$ such that $z + y \leq b\widehat\epsilon x$, for all $s \geq s_*$, where $b$ is a positive constant depending only on $c_0$.
\end{lemma}
\begin{proof}
See Lemma 3.1 in \cite{FK1992}.
\end{proof}
Applying Lemma \ref{lemma x,y,z} to \eqref{system x,y,z}, we see that either $v\rightarrow 0$ exponentially fast, ot there exists $s_*\geq s_0$, such that 
\begin{equation}\label{controll by nutral}
\|v_+\|_{L^2_\rho}+\|v_-\|_{L^2_\rho}+e^{-\frac{\nu}{3} s}\leq b\widehat{\epsilon}\|v_{null}\|_{L^2_\rho}, \forall s\geq s_*,
\end{equation}
for some $b>0$, which concludes the proof of Lemma \ref{lemma eq v+,v-,v null}.
\end{proof}

We now claim the following lemma which is similar to \cite[Proposition 4.1]{FL1993} proved in the case $h\equiv 0$.
\begin{lemma}\label{lemma eq neutral modes}
Assume $v$ does not approach 0 exponentially fast. Then, the neutral mode  satisfies
\begin{equation}\label{eq alpha_2}
\overset{.}{\alpha_2}=\frac{p}{2k}\pi_2(v_{null}^2)+O\left(\alpha_2^{3}\right),
\end{equation}
where $\alpha_2=\pi_2(v)$ and $\pi_2$ is defined in \eqref{def q_beta and k_beta}.
\end{lemma}
\begin{proof}
Applying a Taylor expansion to the order 3 to \eqref{eq v Append}, we have
\begin{equation}
v_s=\mathcal{L}v+\frac{p}{2\kappa}v^2+g(v)+F(v,\nabla v),
\end{equation}
with $|g(v)|\leq C|v|^3$. We project \eqref{eq:v} onto $h_2$, obtaining
$$\overset{.}{\alpha_2}=\frac{p}{2k}\pi_2(v_{null}^2)+\frac{p}{2k}E+\pi_2(F(v,\nabla v)),$$
where
$$E=\frac{p}{2k}\pi_2(v^2-v_{null}^2)+\pi_2(g(v)).$$
From similar computations as in \cite[pages 334 and 335]{FL1993}, we have 
$$|E|\leq Cx^3\leq C |\alpha_2|^3,$$
where $x$ is defined in \eqref{x y z}. Note that for this estimate, we crucially use the fact that $\|v\|_{L^2_\rho}(s)\sim \frac{C}{s}$ (see \eqref{equivalence v null}) in order to control the exponentially small terms in the equation. For the last term, we just use the fact that $|w(s)|+|\nabla w(s)|\leq C$ for all $s\geq s_1$ (From Lemma \ref{bound w} and Proposition \ref{bound nabla w}). Therefore,
$$|F(v,\nabla v)|\leq Ce^{-\nu s}.$$
Thus, using Cauchy-Schwartz inequality, we obtain
$$P_\beta(F(v,\nabla v))=\int F(v,\nabla v)h_2\rho\leq \|F(v,\nabla v)\|_{L^2_\rho}\|h_\beta\|_{L^2_\rho} \leq Ce^{-\nu s}\leq C x^3\leq C |\alpha_2|^3,$$
where we used \eqref{generic behavior} for the third inequality. Therefore, we may conclude the proof of Lemma \ref{lemma eq neutral modes}.
\end{proof}
Now, we have all what is needed for the proof of Lemma \ref{lemma:FL1999}.
\begin{proof}[Proof of Lemma \ref{lemma:FL1999}]
From Lemmas \ref{lemma eq v+,v-,v null} and \ref{lemma eq neutral modes}, one may see that we have\\
\textbf{Case 1: }either $v \rightarrow 0$ exponentially fast as $s \rightarrow +\infty$,\\
\textbf{Case 2:} or 
$$\overset{.}{\alpha_2}=C\alpha_2^2+O(\alpha_2^3).$$
We exclude the first scenario with the following: We use a Taylor expansion on \eqref{equality w and q} for $w_2$ and project onto $h_2$, obtaining
\begin{equation}
\alpha_2(s)+\frac{C(\kappa,p)}{s}+O\left(\frac{1}{s^2}\right)=\pi_2(q_b)(s)+\pi_2(q_e).
\end{equation}
Since $q\in V_A$, for some $A>0$, then, from Propostion \ref{def V_A}, we have
$$\pi_2(q_b)(s)=O\left(\frac{\log (s)}{s^2}\right),$$
and
$$ |\pi_2(q_e)|\leq\int_{y\geq K_0\sqrt{s}}|qk_2\rho|\leq C\frac{e^{-\frac{K_0 s}{4}}}{\sqrt{s}}\int_{y\geq K_0\sqrt{s}}|k_2\sqrt{\rho}|\leq C\frac{e^{-\frac{K_0s}{4}}}{\sqrt{s}}.$$
We obtain then
\begin{equation}\label{equivalence v null}
\alpha_2(s)\sim -\frac{C(\kappa,p)}{s}.
\end{equation}
Since $\alpha_2=\pi_2(v)$. It follows that
$$\|v(s)\|_{L^2_\rho}>\frac{C}{s},$$
for some $C>0$. Therefore, we obtain a contradiction. We have only Case 2 is valid. Then, we obtain the equation in \eqref{eq w_2}. Using \eqref{equivalence v null}, we have the equivalence of $w_{2,2}$ in \eqref{eq w_2}, which concludes the proof of Lemma \ref{lemma:FL1999}.
\end{proof}

\subsection*{Statements:}
\begin{itemize}
    \item  \textbf{Conflict of interest:} The author states that there is no
conflict of interest.\vspace{-0.2cm}
    \item \textbf{Data availability:} We do not analyse or generate any datasets, because our work proceeds within a theoretical and mathematical approach. One can obtain the relevant materials from the references below.
\end{itemize}
\printbibliography

\end{document}